\documentclass[a4paper]{amsart}
\usepackage{amssymb,amsthm,amsmath}
\usepackage{bm}
\usepackage{esint}
\usepackage{mathrsfs}
\usepackage[]{hyperref} %\hypersetup{colorlinks = false, linkbordercolor = myred, citebordercolor = myblue}
\usepackage[noabbrev]{cleveref}
\newtheorem{theorem}{Theorem}[section]

\newtheorem{lemma}[theorem]{Lemma}
\newtheorem{proposition}[theorem]{Proposition}
\newtheorem{corollary}[theorem]{Corollary}
\theoremstyle{definition}

\newtheorem{remark}[theorem]{Remark}
\numberwithin{equation}{section}
\newcommand*{\de}{\partial}
\newcommand*{\N}{\mathbb{N}}
\newcommand*{\R}{\mathbb{R}}

\newcommand*{\eps}{\varepsilon}
\newcommand*{\weak}{\rightharpoonup}

\newcommand*{\loc}{{\text{\upshape{loc}}}}

\newcommand*{\lap}{\Delta}
\newcommand{\res}{\mathbin{\vrule height 1.6ex depth 0pt width 0.13ex\vrule height 0.13ex depth 0pt width 1.3ex}}
\DeclareMathOperator{\supp}{spt}

\author{Federico Franceschini}
\address{ETH, Rämistrasse 101, 8092 Zürich, Switzerland}
\email{federico.franceschini@math.ethz.ch}

\author{Joaquim Serra}
\address{ETH, Rämistrasse 101, 8092 Zürich, Switzerland}
\email{joaquim.serra@math.ethz.ch}

\title{Free boundary partial regularity in the thin obstacle problem}
\begin{document}
\begin{abstract}
    For the thin obstacle problem in $\R^n$, $n\ge2$, we prove that at {\em all} free boundary points, with the exception of a $(n-3)$-dimensional set, the solution differs from its blow-up by higher order corrections. This expansion entails a $C^{1,1}$-type free boundary regularity result, up to a codimension 3 set.
\end{abstract}
\maketitle
\tableofcontents
\section{Introduction}
    \subsection{The thin obstacle problem}
        This paper is concerned with the (local) free boundary regularity for the thin obstacle problem. 
        Namely, let $n\ge 2$, and consider $u\colon B_2 \to \R$ ($B_r\subset \R^n$ denotes the ball of radius $r$ centered at $0$) with $\int_{B_2} |\nabla u|^2 dx <\infty$ and such that the following holds:
        \begin{equation} \label{whtoihwohtw2}
            \begin{split}
            \int_{B_2} |\nabla v|^2 \ge \int_{B_2} |\nabla u|^2& \quad \mbox{for all $v\in H^1(B_2)$}\\ 
            &\hspace{-10mm}\mbox{satisfying $v\ge 0$ on $\{x_n=0\}$ and $v=u$ on $\partial B_2$.}
            \end{split}
        \end{equation}
        
        In this framework the linear space $\{x_n=0\}$ is called the \textit{thin space} and the set
        \begin{equation*}
            \Lambda(u):=\{u=0,x_n=0\}
        \end{equation*}
        is called the \textit{contact set}. When there is no ambiguity, we will write $\Lambda$ instead of $\Lambda(u)$.
        This paper is concerned with the fine structure and regularity properties of boundary of the contact set inside the thin space, which is called the {\em free boundary} (actually in this paper we will choose a slightly different convention and define the free boundary as the set of points of frequency $>1$. This set always contains the boundary of the contact set, but may be larger in some cases).
        
        Now, one can easily see that if $u$ satisfies \eqref{whtoihwohtw2} then $u^{\rm even}(x',x_n) = \frac{1}2 (u(x',x_n) =u(x',x_n))$ and  $u^{\rm odd}(x',x_n) = \frac{1}2 (u(x',x_n) -u(x',x_n))$ also satisfy \eqref{whtoihwohtw2}. In addition, $u^{\rm even}$ has the same contact set as $u$, while $u^{\rm odd}$ is a harmonic odd function. Hence, we may (and do) assume that $u\equiv u^{\rm even}$. In other words 
        $u(x',x_n)= u(x',-x_n)$.
        
        It is well-known (see for instance \cite{AC,ACS}) that $u$ as above is a (weak) solution of the following problem, which is called the {\rm thin obstacle problem}:
        \begin{equation}\label{eq:thinobstacle}
            \begin{cases}
                    \lap u \leq 0 & \text{ in }B_2,\\
                    u \geq 0 & \text{ in } \{x_n=0\}\cap B_2,\\
                    \lap u = 0 & \text{ in } (\{u>0\} \cup \{x_n \neq 0\}) \cap B_2,\\
                    u(x',x_n)=u(x',-x_n) & \text{ for all }(x',x_n)\in B_2.
            \end{cases}
        \end{equation}

        \subsection{Motivations and connections to other problems}
        
        Problems of the type \eqref{whtoihwohtw2} arise in linear elasticity (the Signorini problem) when describing the equilibrium configuration of an elastic body, which rests on a rigid frictionless horizontal plane and subject only to its mass forces. It also arises in optimal control of temperature across a surface and in the modelling of semipermeable membranes. See for instance \cite{ACS,Fe,Se} and references therein.
        
        Since the 2019 paper \cite{FS},  there is new main motivational driver for \eqref{whtoihwohtw2} related to the fine study of singular points in the obstacle and Stefan problems. See \cite{FS,FRS1, FRS2}.

        \subsection{Almgren's mononotonicity formula and its first consequences}
        A key contribution from \cite{ACS} was to bring to light a strong analogy between \eqref{whtoihwohtw2} and Almgren's $Q$-valued harmonic functions\footnote {Appearing in his theory of area-minimizing integer rectifiable currents.}. In particular \cite{ACS} showed that Almgren's frequency had a central role  in the study \eqref{whtoihwohtw2}, as explained next.
        
        Throughout the paper we will use the following functionals\footnote{These functionals have been often considered in the previous literature on the thin obstacle problem.}:
        \begin{equation*}\label{eq:functionals}
            H_\mu(r,w):=r^{-2\mu}\int_{\de B_1}w_r^2, \ D(r,w):=\int_{B_1}|\nabla w_r|^2 \ \text{ and }\  \phi(r,w):=\frac{D(r,u)}{H_0(r,w)},
        \end{equation*}
        defined for all $r\in(0,1)$ and $\mu\in\R$.  Whenever no ambiguity is possible $d\mathcal H^n$ and $d\mathcal H^{n-1}$ are omitted, respectively, from solid integrals and boundary integrals. We often write $H(r,w)=H_0(r,w)$. 
        
        A key observation from \cite{ACS} was that for every $x_\circ\in\Lambda(u)$  Almgren's monotonicity formula holds for $w = u(x_\circ +\,\cdot\,)$. More precisely, 
        \begin{equation}\label{whgiohwiohw}
            \frac{d}{dr}\phi(r, w)) \ge 2 r^{3-2n} \frac{\big\{\int_{\partial B_r} w_\nu^2\int_{\partial B_r} w^2  -\big(\int_{\partial B_r}ww_\nu\big)^2 \big\}}{H(r, w)^2 } \ge 0.
        \end{equation} 
        Hence, for every $x_\circ\in\Lambda(u)$, the function $r\mapsto \phi(r,u(x_\circ+\cdot))$ is nonnegative and nondecreasing, so it is well-defined
        \[
            \phi(0+,u(x_\circ+\cdot)) : = \lim_{r\to 0+}\phi(r,u(x_\circ+\cdot)).
        \]
        This induces a natural stratification in $\Lambda = \cup_{\lambda\ge 0} \Lambda_\lambda$, where
        \begin{equation*}
            \Lambda_\lambda(u):=\{x_\circ\in\Lambda(u) : \phi(0^+,u(x_\circ+\cdot))=\lambda\},
        \end{equation*}
        and we say that a point $x_\circ$  belonging to $\Lambda_\lambda(u)$ is {\it a point of frequency} $\lambda$. Similarly, given $I \subset [0, \infty)$ let us denote
        \[
        \Lambda_I : = \cup_{\lambda\in I} \Lambda_\Lambda.
        \]
        
        The monotonicity formula \eqref{whgiohwiohw}  yields the existence and homogeneity of blow-ups. More precisely, given $x_\circ\in\Lambda(u)$ one defines {\em blow-ups at $x_0$} as accumulation points in $L^2_{\rm loc}(\R^n)$ of the type
        \begin{equation}\label{whtiowhio657} 
        q : = \lim_{r_k \downarrow 0}  \frac{u(x_\circ +r_k \,\cdot\,)}{\|u(x_\circ + r_k\,\cdot\,)\|_{L^2(\partial B_1)}}.
        \end{equation}

        It is not difficult to show using  \eqref{whgiohwiohw} --- see for instance \cite{ACS} --- that for every sequence $r_\ell\downarrow 0$ there is a subsequence such that the limit \eqref{whtiowhio657} exists in $L^2_{\rm loc}$. Moreover,  all blow-up profiles $q$ at a given point $x_\circ\in \Lambda_\lambda$
        are $\lambda$-homogeneous\footnote{In this context $\lambda$-homogeneous will always mean positively $\lambda$-homogeneous, that is,  $q(tx) = t^\lambda q(x)$ for all $x\in \R^n$  and $t>0$.} solutions 
        of \eqref{eq:thinobstacle}.
        
        We emphasize that blow-ups profiles $q$ at a given point $x_\circ$ may, a priori, depend on the (sub)sequence $r_k \downarrow 0$, and hence a given point might have multiple blow-up profiles. 
        
        \subsection{The question of free boundary regularity}
        One of the main contributions of Athanasopoulous, Caffarelli, and Salsa in \cite{ACS} was to provide a very detailed understanding of points of frequency less than two. More precisely, they proved that
        \[
        \Lambda_{[0,2)} = \Lambda_1 \cup \Lambda_{3/2}. 
        \]
        Moreover, it follows --- using a standard Federer-type ``dimension reduction'' argument based Almgren's  frequency formula --- that the set 
        \[\Lambda\setminus\Lambda_1 = \Lambda_{3/2}\cup\Lambda_{[2,\infty)}\]
        is closed and has dimension $n-2$. We will call this set the {\em free boundary} as it always contains (and typically equals) the (relative) topological boundary of the contact set $\Lambda(u)$\footnote{Some authors would exclude from the free boundary odd frequency points which may lie in the interior of $\Lambda$, but there is of course no harm in including such points in our analysis.}. The main question is:
        \begin{quote}
            {\em Does the free boundary in the thin obstacle problem enjoy $C^\infty$ regularity, perhaps after removing lower dimensional sets? }
        \end{quote}
        
        A first very important contribution in this direction was given by \cite{ACS}. They showed that $\Lambda_{3/2}$ is a $(n-2)$-dimensional submanifold of class $C^{1,\alpha}$, for some $\alpha>0$. Later in \cite{KPS} regularity was bootstrapped to analytic (see also \cite{DS1,DS2}, where $C^\infty$ regularity is proven with a different method). Unluckily, in general the set $\Lambda_{[2,\infty)}$ is not small, nor the set $\Lambda_{3/2}$ is necessarily there. There are known examples of solutions $v$ of \eqref{eq:thinobstacle} for which $\Lambda_{3/2}(v)=\emptyset$ and $\Lambda(v)=\{x_n=0,x_{n-1}\leq 0\}$. Thus there are ``genuine'' smooth free boundaries completely made of points in $\Lambda_{[2,\infty)}$.
        
        Although several important works (see Section \ref{sec:prevresults} for a detailed account of these) had investigated the fine structure of certain subsets of $\Lambda_{[2,\infty)}$, all these results were only applicable to rather exceptional subsets of $\Lambda_{[2,\infty)}$  (and in particular they did not cover the ``genuine'' examples referred above).
        
        The results of this paper, described in the next section, establish for the first time enhanced regularity properties of the whole free boundary except for a lower ($n-3$) dimensional set of ``bad'' points.

        \subsection{Main result} In this paper we are concerned with establishing that all points for $\Lambda_{[2,\infty)}$ outside of a $(n-3)$-dimensional satisfied strongly enhanced regularity properties: in particular they can be covered $(n-2)$-manifolds of class $C^{1,1}$ . 
        The question of how to improve this $C^{1,1}$ regularity to $C^\infty$ by using more technical versions of the new methods introduced in this paper is deferred to a future work.
        Here we focus for simplicity in proving the result below, which follows by applying our main new ideas in a clean, non-technical, setting.
        
        Our main result is the following
        \begin{theorem}\label{thm:mainintro}
            Let $u$ solve the thin obstacle problem \eqref{eq:thinobstacle} and let $\Lambda:=\{u=0,x_n=0\}$ be its contact set, $n\geq 3$. Then there is $\Sigma\subset\Lambda$ such that
            \begin{itemize}
                \item the Hausdorff dimension of $\Sigma$ is at most equal to $n-3$,
                \item if $x_\circ \in\Lambda\setminus \Sigma$ and $u$ has frequency $\lambda$ at $x_\circ$, then $ \lambda\in\{\tfrac{3}{2},\tfrac{7}{2},\tfrac{11}{2},\ldots\}\cup\{1,2,3,4,\ldots\}$ and 
                \begin{equation}\label{eq:expansionintro}
                    u(x_\circ+y)=q(y)+O(|y|^{\lambda +1 })\quad \text{ as }|y|\downarrow 0,
                \end{equation}
                where $q$ is a 2D, $\lambda$-homogeneous and non-zero solution of the thin obstacle problem\footnote{Such solutions are classified and given by explicit formulas (see \Cref{lem:2Dsolutionsclassification}).  }. 
            \end{itemize}
        \end{theorem}
        
        We remark that uniqueness of blow-ups\footnote{In the sense that they it does not depend on the subsequence.} --- up to a co-dimension 3 set --- was not known, while our result implies it immediately.
        
        The previous theorem entails the following
        \begin{corollary}\label{cor:covering}
            The free boundary  of the thin obstacle problem in $\R^n$ is covered by countably many $(n-2)$-submanifolds of class $C^{1,1}$, with the exception of a set of Hausdorff dimension at most equal to $n-3$.
        \end{corollary}
        As said above, in future work we plan to address the necessary modifications of the methods of this paper in order to enhance the $C^{1,1}$ regularity to  $C^\infty$. This will be done by obtaining higher order expansions in Theorem \ref{thm:mainintro}. In this direction, it is interesting to remark that in the present paper we are really obtaining a more precise result than Theorem \ref{thm:mainintro}. Namely, we prove 
        \begin{equation}\label{eq:expansionintrobis}
                    u(x_\circ+y)=q(y)+O(|y|^{\lambda +m })\quad \ \text{ as }|y|\downarrow 0,
        \end{equation}
        where $m\ge 1$ is an integer such that, for some $r_k\downarrow 0$ 
        \[
        \frac{u(x_\circ+r_ky)-q(r_ky)}{r_k^{\lambda + m}} \to  Q
        \]
        where $Q$ is a $(\lambda+m)$-homogeneous harmonic function which vanishes on the half-space $\{x_n=0\}\cap \{q=0\}$. Such $Q$ will give precisely the next order in the expansion, and we expect to be able to apply similar (but significantly more technical) methods as in the proof of Theorem \ref{thm:mainintro}, combined with ideas from \cite{FZ}, to $u-q-Q$ (instead of $u-q$).
            
   \subsection{A survey of previous results}\label{sec:prevresults}
       As said above, several important works investigated the fine structure of $\Lambda_{[2,\infty)}$, obtaining very interesting results which nevertheless can only be applied to ``exceptional'' subsets $\Lambda_{[2,\infty)}$. 
       In order to explain them, it is helpful to further split $\Lambda_{[2,\infty)}$ into special subsets:
       \begin{equation*}
           \Lambda_{[2,\infty)}=\Lambda_{2+2\N}\cup \Lambda_{3+ 2\N}\cup\Lambda_{ \frac 3 2 + 2\N }\cup \Lambda_{\rm other},
       \end{equation*}
       the union being disjoint (here $\N=\{0,1,2,\dots\}$). The common feature of these results was establishing estimates of the following type: if $x_\circ\in B_{1/2}$ belongs to one of these special sets and has frequency $\lambda$ then, for some $\lambda$-homogeneous solution $q_{x_\circ}$ and for some modulus of continuity $\omega_{x_\circ}$, it holds
       \begin{equation}\label{eq:expansionoverview}
           \sup_{B_r}|u(x_\circ+\cdot)-q_{x_\circ}|\leq r^\lambda \omega_{x_\circ}(r)\|u\|_{L^2(B_2)}\quad \text{ for all }r\in(0,1).
       \end{equation}
    More precisely, the current picture is:
    \begin{itemize}
        \item The set $\Lambda_{2+2\N}$ (points with even frequency) was investigated for the first time in \cite{GP}, using Monneau-type monotonicity formulas, it was showed that \eqref{eq:expansionoverview} holds at every $x_\circ\in\Lambda_{2\N}$, for a suitable $\lambda$-homogeneous harmonic polynomial $q_{x_\circ}$ and abstract modulus of continuity. 
        
        This expansion was improved in \cite{CSV1} (and further in \cite{SY2}). In both cases, using the ``epiperimetric inequality approach'' introduced by Weiss \cite{W}, the authors showed that $\omega_{x_\circ}(r)\leq C(n,\lambda)|\log r|^{-\alpha}$, for an explicit exponent $\alpha=\alpha(n,\lambda)>0$.
        
        In \cite{FeJ}, using the the ideas conceived in \cite{FS} for the obstacle problem (and which are also applicable to the analysis of $\Lambda_{2\N +2}$ in the thin obstacle problem since  $q$ is always harmonic for those points), it was proved that actually $\omega_{x_\circ}(r)=O(r)$ for all $x_\circ\in\Lambda_{2+2\N}\setminus E$, where $E$ has Hausdorff dimension at most $n-3$ (see \cite[Proposition 1.10 and Remark 1.12]{FeJ}).
        
        Points in $\Lambda_{2+2\N}$ have sometimes been called {\it singular}, in analogy with the classical obstacle problem, since the contact set $\Lambda$ has vanishing $\mathcal H^{n-1}$ density at such points (see \cite{GP}). Intuitively speaking, they are not fully genuine free boundary points as the thin obstacle is tangent to the solution and ``does not exert force'' on it.  More precisely, we expect that, outside of a $(n-3)$-dimensional set,  the tangency between the thin obstacle and the solution is at infinite order. 
        
        In a similar direction, examples exhibiting even homogeneity points can be constructed ``artificially'' in the following way: take any harmonic function $v$ which is positive in the thin space and $x'\mapsto v(x',0)$ has a local minimum $\mu\geq 0$ at $x'=0$. Then $v-\mu$ solves \eqref{eq:thinobstacle} in a small ball around this local minimum point and $$0\in\{v=\mu,x_n=0\}\subset\Lambda_{2+2\N}(v-\mu).$$
        \item $\Lambda_{3+2\N}$ consists of points with odd frequency ($\geq 3$). In all the known examples $\Lambda_{3+2\N}$ lies in the interior of $\Lambda$. The interest in such points is motivated mainly because they may arise in the 3$^{\rm{rd}}$ order blow-up analysis of solutions of the classical obstacle problem (see \cite{FS,FRS1}). In \cite[Appendix B]{FRS1} it was proved that \eqref{eq:expansionoverview} holds {at every} $x_\circ\in\Lambda_{3+2\N}$, where $q_{x_\circ}$ is (the even reflection of) a $\lambda$-homogeneous harmonic polynomial.
        
        Very recently, a striking result from \cite{SY2}, was that actually $\omega_{x_\circ}=O(r^{\lambda+\alpha})$ for some universal $\alpha>0$ at every point. This was proved again with the ``epiperimetric inequality'' approach.
        
        \item $\Lambda_{\frac32+2\N}$ consists of points with frequency in $\lambda\in\{\frac32,\frac72,\ldots\}$.         Despite their importance in the free boundary analysis, very little is known about them in general.
        
        For every such $\lambda$ it is simple to construct a solution $v$ of \eqref{eq:thinobstacle} such that $v(x',0)=(x_{n-1})_+^\lambda$, thus giving rise to a perfectly genuine free boundary: $\{x_{n-1}=x_n=0\}$. Intuitively, such points look like the degenerate version of the regular points $\Lambda_{3/2}$, but it is not clear how such degeneracy might affect the free boundary regularity.
        
        Blow-ups (i.e., $\lambda$-homogeneous solutions) are classified only in the case $n=2$ or if one prescribes their contact set to be (up to rotations) $\{x_{n-1}\leq0,x_n=0\}$. Let us call $\mathcal F$ this family of blow-ups. It is not known if other blow-ups might exist.
         
        In the case $n=2$ expansion \eqref{eq:expansionoverview} holds at every point with $\omega_{x_\circ}(r)\leq C(\lambda)r^{\lambda+\alpha}$ for some universal $\alpha>0$ and ``the'' 2D $\lambda$-homogeneous solution $q_{x_\circ}$ (see \cite[Theorem 3]{CSV1}).
        
        Savin and Yu recently proved in \cite{SY2} that, in the case $n=3$ and $\lambda=7/2$, if one already knows that some blow-up lies in $\mathcal F$, then \eqref{eq:expansionoverview} in fact holds for some $q_{x_\circ}\in\mathcal F$ and some weaker remainder\footnote{Interestingly, the rate of convergence depends on $q_{x_\circ}$. If $q_{x_\circ}$ is the 2D solution, then $\omega_{x_\circ}=O(r^\lambda|\log r|^{-\alpha})$, otherwise $\omega_{x_\circ}=O(r^{\lambda+\alpha})$. In both cases $\alpha$ is universal.}.
        
        Using that the whole free boundary is known to be $\mathcal H^{n-2}$-rectifiable (see \cite{FoS1,FoS2}), in \cite{CSV2} it is showed that, $\mathcal H^{n-2}$-a.e. in $\Lambda_{\frac32+2\N}$, \eqref{eq:expansionoverview} holds for some 2D blow-up $q_{x_\circ}$.
        
        The main new contribution of the present paper is to provide some general result about these points, up to codimension 3.
        
        \item $\Lambda_{\rm other}$ is defined as the set of points whose frequency is neither an integer nor belongs to $\{\frac32,\frac72, \frac {11}2\ldots\}$. Examples of solutions such that $\Lambda_{\rm other}\neq \emptyset$ are not known. This set (if it exists) is anyway known to be $(n-3)$-dimensional at most, as a consequence of the monotonicity of the frequency \eqref{whgiohwiohw} --- see for instance \cite{FS} (we also re-prove this fact here in \Cref{prop:dimredotherfreq}).
    \end{itemize}
    
    \subsection{Future directions and open questions} In view of the results of this paper, it seems  plausible that some open questions on related problems can be now tackled with similar methods. For instance:
    \begin{enumerate}
        \item[1.] For ``generic boundary data'' (as in \cite{FRS1}), the singular set in  the obstacle problem in $\R^n$ has dimension $n-5$.  
        \item[2.] For {\em every} solution of the obstacle problem in a ball of $\R^3$, the singular set can be covered by (countably many) $C^\infty$ curves and surfaces, after removing a 0-dimensional set.
    \end{enumerate}
     On the other hand, our results also raises natural (and seemingly difficult) questions, for example:   
    \begin{enumerate}
        \item[3.] For $\lambda\in 7/2+2\N$ is the set of ``good'' points in $\Lambda_\lambda$ where  \eqref{eq:expansionintro} holds relatively open, or there are counterexamples where it could be, say, a Cantor-like set?
    \end{enumerate}

    \subsection{Acknowledgements} Both authors have been supported by Swiss NSF Ambizione Grant PZ00P2 180042 and by the European Research Council (ERC) under the Grant Agreement No 948029.

\section{Explanation of the proofs and organization of the paper}\label{sec:overview}
    \subsection{}
    Let $u$ be a solution of \eqref{eq:thinobstacle} such that $0$ is a contact point with frequency $\lambda$, that is $0\in\Lambda_\lambda$. As explained above the most interesting and less understood situation happens when $\lambda\in\{\frac32,\frac72,\frac{11}2,\ldots\}$; let us work out this case. 
    It is known (and we will re-prove in \Cref{sec:2Dblowups}) that blow-ups $q$ at all points,  except for a $(n-3)$-dimensional set, are  translation invariant along $n-2$ directions (in addition to $\lambda$-homogeneous); in other words, up to rotation $q(x', x_{n-1},x_n)=q(x_{n-1}, x_n)$ is a ``2D solution''.  Such solutions are classified: they satisfy \[q(x', x_{n-1},0)=\tau (x_{n-1})_+^\lambda,\] for some $\tau>0$ and all $(x',x_{n-1})$ in the thin space. The core of this work consists of two {\em independent} steps
    
        \vspace{3pt}
        \noindent\textbf{- Step 1:} For all $x_\circ\in\Lambda_\lambda$ except for a $(n-3)$-dimensional set, there exist $\beta=\beta(x_\circ)>0$ and a 2D $\lambda$-homogeneous solution $q=q_{x_\circ}$, such that \[\sup_{B_1}|u(x_\circ+r\cdot)-q(r\cdot)|=O(r^{\lambda+\beta})\text{ as }r\downarrow 0.\]
        
        %\vspace{3pt}
        \noindent\textbf{- Step 2:} Assume that as some point we already know that $|(u-q)(r\cdot)|=O(r^{\lambda+\beta})$, for some $\beta>0$ and some 2D $\lambda$-homogeneous solution $q$. Then (a truncated version) of Almgren's  frequency function applied to $u-q$ is (approximately) monotone. Using this fact we classify the next possible order in the expansion, and in particular improve the error from $O(r^{\lambda+\beta})$ to $O(r^{\lambda+1})$.
        
        \vspace{3pt}
    
    The first step is proved in \Cref{sec:accelerateddecay}, and the second in \Cref{sec:frequency}. 
    It is clear that \Cref{thm:mainintro} is essentially the combination of these two steps, so let's explain briefly how they work.
    \subsection{Proof of the first step}
        This step is clearly inspired by the methods in \cite{FRS2} (however, as it is clear by comparing the two proofs, several new ideas have been needed to make the same general strategy work in the present setting). The core of this proof is \Cref{prop:dichotomy}.
        It is natural to introduce at each scale $r\in(0,2)$ the quantity
        \[
            A(r,u):=\min_{q} \|(u-q)(r\cdot)\|_{L^2(\de B_1)},
        \]
        where the minimum is taken among all 2D, $\lambda$-homogeneous solutions. We will work in the assumption that $A(1,u)\leq \eta \|u\|_{L^2(\de B_1)}$ for some small $\eta$ which we will choose in the end. Under such assumption we first prove a \textit{weak} unconditional decay of the function $A(\cdot,u)$: for any arbitrary small $\gamma>0$ we have
        \[
            A(2^{-k-1},u)\leq (1/2)^{\lambda-\gamma}A(2^{-k},u)\quad \text{ for all }k\geq0.
        \]
        This decay is \textit{weak} in the sense that the critical exponent to break would be $\lambda$, nevertheless it will be useful. Its proof is based on a Monneau-type monotonicity formula (cf. \Cref{subsec:monneau}).
        
        Next, always under the smallness of $\eta$ assumption, we establish a dichotomy at points in $\Lambda_\lambda$ at every dyadic scale $B_{2^{-k}}$. Fix any $\eps>0$, then for each $k\geq 0$ we show that
        \begin{itemize}
            \item[ \underline{either}] inside $B_{2^{-k}}$, the set $\Lambda_\lambda(u)$ lies in an $(\eps/2^k)$-tubular neighborhood of some $n-3$ dimensional linear space. With a slight abuse of language, this can be summed up in the sentence ``$\Lambda_\lambda(u)$ is $\eps$-flat at scale $2^{-k}$''.
            \item[\underline{or}] $A(\cdot,u)$ has a \textit{strong} decay from the scale $k$ to the scale $k+1$, that is
            \[
                A(2^{-k-1},u)\leq (1/2)^{\lambda+\sigma}A(2^{-k},u),
            \]
            where now $\sigma=\sigma(n,\lambda,\eps)>0$ is independent from $\eta$. Since a strong decay happens, we will call these scales ``good scales''.
        \end{itemize}
        
        Let's now fix any point $x_\circ\in\Lambda_\lambda(u)$ and run this dichotomy from scale $k=1$ to scale $k=m\gg1$: at the scales where $\Lambda(u)$ is $\eps$-flat we use the {weak} decay of $A$, while at the good scales we use the {strong} decay. We get
        \begin{equation}\label{eq:iterativedecayoverview}
            A(2^{-m},u(x_{\circ}+\cdot)\leq (1/2)^{m\lambda+N\sigma-(m-N)\gamma}\eta\|u(x_\circ+\cdot)\|_{L^2(\de B_1)},
        \end{equation}
        where $N=N(m,\eps,x_\circ)$ is the number of \textit{good} scales among $\{1,1/2,\ldots,2^{-m}\}$.
        
        Now inspecting \cref{eq:iterativedecayoverview}, at each $x_\circ$, it is natural to look at the ``asymptotic average number of good scales'', namely the number
        \begin{equation*}
            p(\eps,x_\circ):=\liminf_m\frac{N(m,\eps,x_\circ)}{m} \ \in[0,1].
        \end{equation*}
        If there exist some $\eps_\circ>0$ such that $p(\eps_\circ,x_\circ)>0$, using \eqref{eq:iterativedecayoverview} one formally gets $$\frac{A(2^{-k},u(x_\circ+\cdot))}{\eta\|u(x_\circ+\cdot)\|_{L^2(\de B_1)}}``\lesssim" (1/2)^{k(\lambda+p\sigma-(1-p)\gamma)}\leq  (1/2)^{k(\lambda+\beta)}\text{ for all }k\geq 0,$$
        where we chose $\gamma=\gamma(p)$ so that $\beta=\beta(p)$ is a positive number.
        
        Assume instead that for all $\eps>0$ we have $p(\eps,x_\circ)=0$. This means that, for all $\eps>0$, $\Lambda_\lambda(u)$ is $\eps$-flat, in average, at all scales. As intuition suggests, a GMT argument shows that points with this property have Hausdorff dimension at most $n-3$ (cf. \Cref{prop:dimensionreduction}).
        
        We remark that this method of proof simplifies greatly in the case $\lambda=3/2$, yielding a new simple proof of the regularity of $\Lambda_{3/2}$.
    \subsection{Proof of the second step}
        This step holds at each point so we assume $x_\circ=0$. We set $w_r:=(u-q)(r\cdot)$ and assume that for some $\beta,r_\circ>0$ we have
        \begin{equation}\label{eq:decayassumptionoverview}
            \sup_{B_r}|w|\leq r^{\lambda+\beta} \text { for all } r\in(0,r_\circ).
        \end{equation}
        We prove that we can improve such decay from $r^{\lambda+\beta}$ to $r^{\lambda+1}$ {\em always}, just by studying the frequency function on $w$ and the problem solved by its blow-ups. Namely, we will prove that, for any fixed large $\gamma\geq 0$ the function
        \begin{equation}\label{eq:monotnicityoverview}
            (0,r_\circ)\ni r\mapsto \frac{D(r,w)+\gamma r^{2\gamma}}{H(r,w)+r^{2\gamma}}+Cr^{\alpha\beta/\lambda},
        \end{equation}
        is increasing, where $\alpha=\alpha(n)>0$.
        This ``truncated frequency function'' was already introduced in \cite{FRS1} in the context of the obstacle problem. The striking novely here is that can prove its monotonicity for $w=u-q$, where $q$ is a 2D solution of homogeneity in $2\N+\frac 3 2$, and hence not harmonic. All the previous proof relied (both in Almgren's original works or \cite{FRS1}) used in a crucial way harmonicity of $q$, and with our approach here we see for the first time that harmonicity is not as essential as it seemed!
        
        Fix $\gamma\ge \lambda+1$ and let $\mu$ be the value of the truncated frequency function as $r\downarrow 0$. The monotonicity of the frequency is a very powerful tool, since it allows to do a second blow-up,  showing that for some sequence $r_\ell\downarrow 0$ we have
        \begin{equation*}
            \lim_k \frac{w_{r_\ell}}{\|w_{r_\ell}\|_{L^2(\de B_1)}}=Q,
        \end{equation*}
        in the $L^2_\loc(\R^n)$ topology, where $Q$ is $\mu$-homogeneous. By studying the (linear) problem satisfied by $Q$ we can show that  $\mu\in \N+\frac 1 2 \subset \mathbb Z +\lambda $. But then, using  \eqref{eq:decayassumptionoverview} we obtain $\mu\geq \lambda+\beta$, so it must be $\mu\geq \lambda+1$, which was essentially our goal.
        
        The (completely new) proof of the monotonicity \eqref{eq:monotnicityoverview} relies on several ingredients, an important one is the following ``nonlinear'' estimate (cf. \Cref{lem:estimatewlapw,lem:keyestimateformonotonicity})
        \begin{align}\label{eq:wrlapwroverview}
            \int_{B_1}|w_r\lap w_r|\leq Cr^{\alpha\beta/\lambda}\|w_r\|^2_{L^2(B_2\setminus B_1)}.
        \end{align}
        The proof of this estimate turns out to be quite short (but nontrivial). The main difficulty is that points of half-integral frequency display a two-fold behaviour: in the half space where $q>0$, they behave like points with even frequency, while in the other half space, where $q\equiv 0$, they behave like points with odd frequency. For points of even frequency the Almgren's frequency was known to be exactly monotone, but for points of odd frequency, even after some neat simplifications, some errors appear, thanks to estimate \eqref{eq:wrlapwroverview} we are able to keep such terms under control.
        \subsection{Organization of the paper} In \Cref{sec:preliminaries} we collect some some preliminary estimates which we will use throughout the paper, \Cref{subsec:elliptic,subsec:estimatesdifference} will be particularly important. In \Cref{sec:2Dblowups} we carry out a preliminary dimension reduction to show that blow-ups are 2D (such result is well-known in its ``qualitative'' version, but since we need a ``quantitative'' version we re-prove it for the reader's convenience). In \Cref{sec:accelerateddecay} we prove Step 1. In \Cref{sec:frequency} we prove Step 2. In \Cref{sec:proofmainresults} we finally give the proof of \Cref{thm:mainintro} and \Cref{cor:covering}.
\section{Preliminaries}\label{sec:preliminaries}
    \subsection{Notation}\label{subsec:notations}
        We work in $\R^n$ where $n\geq 2$. When $f$ is a function defined on $\R^n$ and $r\in\R$, we set $f_r(x):=f(rx)$. We remark that $\nabla u_r= r (\nabla u)_r$. We denote by $O(n-1)$ the group of $n\times n$ orthogonal matrices which fix $e_n$. We will say that a function is increasing also in the case in which it is just nondecreasing. $\mathcal C_\delta:=\{x_{n-1}^2 + x_n^2\leq \delta^2\}$ and $\mathcal N_\delta:=\{(x_{n-1})_+^2 + x_n^2 \leq \delta^2\}$. $\mathcal{X}_E: \R^n\to \{0,1\}$ denotes indicator function of the set $E\subset \R^n$.
   
   \subsection{A useful lemma} 
        As a standard consequence (see for instance the proof of Lemma 2.6 in \cite{FS}) of the monotonicity of the frequency function we have the following:
        \begin{lemma}\label{lem:Hlambdau}
            If $u$ solves \eqref{eq:thinobstacle} and $0\in\Lambda_\lambda(u)$ for some $\lambda\geq 0$, then
            \begin{equation*}
                \left(\frac R r \right)^{\lambda} \leq \frac{H_0(R,u)}{H_0(r,u)}\leq \left(\frac R r \right)^{\lambda+\phi(R,u)}\text{ for all }0<r<R<2.
            \end{equation*}
        \end{lemma}
    \subsection{Homogeneous solutions in 2D}    
        Homogeneous solutions of \eqref{eq:thinobstacle} are classified in 2D. We start defining the following $\lambda$-homogeneous functions
        \begin{equation}\label{eq:psilambdadef}
            \psi_\lambda(x,y):= a
            \begin{cases}
                \text{Re}(\sqrt{x+i|y|}^{2\lambda}) & \text{ if }\lambda\in \frac{1}{2}+\N,\\
                -\text{Im}({x+i|y|})^{\lambda} & \text{ if }\lambda\in 1+2\N,\\
                \text{Re}({x+i|y|})^{\lambda} & \text{ if }\lambda\in 2\N,
            \end{cases} 
        \end{equation}
        where $\sqrt{\cdot}\colon\mathbb C\setminus\{y<0,x=0\}\to\R$ is the branch of the square root such that $\sqrt{1}=1$. The constant $a=a(n,\lambda)>0$ is chosen in such a way that $\|\psi_\lambda\|_{L^2(\de B_1)}=1$. It is easy to check that $\psi_\lambda$ is continuous and non-negative on $\{y=0\}$ for all $\lambda$'s for which we defined it. More precisely, on the thin space we have
        \begin{equation}\label{eq:signpsilambda}
            \psi_\lambda(x,0)= 
            a\begin{cases}
                (x)_+^\lambda & \text{ if }\lambda\in \frac{1}{2}+\N,%\\
                %(x)_+^\lambda & \text{ if }\lambda\in \frac{3}{2}+2\N,
                \\
                0 & \text{ if } \lambda\in 1+2\N,\\
                x^\lambda & \text{ if }\lambda\in 2\N.
            \end{cases}    
        \end{equation}
        The distributional laplacian of $\psi_\lambda$ is given by
        \begin{equation*}
            \lap\psi_\lambda = 2\frac{\de \psi_\lambda}{\de y}(\cdot,0+)\ \mathcal{H}^{n-1}\res{\{y=0\}},
        \end{equation*}
        more explicitly, for some constant $
        c=c(n,\lambda)>0$ we have
        \begin{equation}\label{eq:laplacianpsilambda}
            \lap\psi_\lambda(x,y)=-c
            \begin{cases}
             -(x)_-^{\lambda-1}\delta_0(y) & \text{ if }\lambda\in \frac{1}{2}+2\N,\\
             (x)_-^{\lambda-1}\delta_0(y) & \text{ if }\lambda\in \frac{3}{2}+2\N,\\
                x^{\lambda-1}\delta_0(y) & \text{ if }\lambda\in 1+2\N,\\
                0 & \text{ if }\lambda\in 2\N.
            \end{cases}
        \end{equation}
        We point out that, for $\lambda\in\frac32+\N$, we have
        \begin{equation*}
            \frac{\de\psi_\lambda}{\de x}=b\psi_{\lambda-1},
        \end{equation*}
        for some $b=b(n,\lambda)>0$.
        \begin{lemma}\label{lem:2Dsolutionsclassification}
            Suppose $u$ solves \eqref{eq:thinobstacle} in $\R^2$ and $x\cdot \nabla u =\lambda u$ for some $\lambda\geq 0$. Then $\lambda\in\{1,2,3,\ldots\}\cup \{\frac{3}{2},\frac{7}{2}, \frac{11}{2}\ldots\}$ and $u=\tau\psi_\lambda$ for some $\tau\geq 0$.
        \end{lemma}
        \begin{proof}
            One can explicitly solve the ODE for the function $\theta\mapsto u(\cos\theta,\sin\theta)$. 
        \end{proof}
    \subsection{The ``linearized problem''}
        We introduce the following ``linearized problem'' which will be satisfied (approximately) by $w_r(x): = (u-q)(rx)$, whenever $q$ is a 2D blow-up of homogeneity $\lambda\in \frac 3 2 + 2\N$.
        
        For $R>0$ let us consider
        \begin{equation}\label{eq:thinobstaclelinearized}
            \begin{cases}
                    \lap w = 0 & \text{ in }B_R\setminus\mathcal N_0,\\
                    w = 0 & \text{ in } B_R\cap\mathcal N_0,\\
                    w(x',x_n)=w(x',-x_n) & \text{ for all }(x',x_n)\in B_R,
            \end{cases}
        \end{equation}
        (recall that $\mathcal N_\delta=\{(x_{n-1})_+^2 + x_n^2 \leq \delta^2\}$).
        We also define for each $\lambda \geq 0$ the \textit{linear} space of homogeneous solutions of \eqref{eq:thinobstaclelinearized}.
        \begin{equation}\label{eq:defHcallambda}
            \mathcal H_\lambda:=\left\{v\in H^1_\loc(\R^n) : v=0\text{ on }\mathcal N_0, \lap v = 0\text{ in }\R^n\setminus \mathcal N_0, x\cdot \nabla v=\lambda v\right\},
        \end{equation}
        it is easily checked that for each $\lambda\in\frac1 2 +\N$ we have $\psi_\lambda\in\mathcal H_\lambda$. Moreover, standard separation of variables (radial, spherical) and spectral theory give
        \begin{proposition}\label{prop:sphericalharmonics}
            The following facts hold:
            \begin{enumerate}
                \item $\mathcal H_\lambda$ is non-trivial only if and only $\lambda\in\frac 1 2 + \N$, in which case $\mathcal{H}_\lambda$ has finite dimension $d_\lambda$ (and solutions are explicit\footnote{It is well-known that the homogeneity $\lambda= \frac 1 2$ corresponds to the first eigenvalue of the problem on the sphere. Hence, $\mathcal H_\lambda$ is one dimensional (and generated by $\psi_{1/2}$) . It is then possible  to compute explicitly the solutions spanning $\mathcal H_\lambda$ for $\lambda\ge \frac 3 2$ by exploiting  the following fact: if $v\in \mathcal H_\lambda$ and $\lambda\ge \frac 3 2$,  then $\partial_j v\in \mathcal H_{\lambda-1}$ for all $j\in\{1,2,\dots, n-2\}$.}). 
                \item Each  and $\mathcal{H}_\lambda\perp\mathcal H_\lambda'$ in the $L^2(\de B_1)$ scalar product, provided $\lambda\neq \lambda'$.
                \item If $w\in H^1(B_1)$ solves \eqref{eq:thinobstaclelinearized} for $R=1$ then following expansion in spherical harmonics holds in the $H^1(B_1)$ topology
                \begin{equation*}
                    w=\sum_{\lambda\in \N+1/2}\sum_{m=1}^{d_\lambda} c_{\lambda,m}\psi_{\lambda,m},
                \end{equation*}
                where $\{\psi_{\lambda,1},\ldots,\psi_{\lambda,d_\lambda}\}$ is an orthonormal basis of $\mathcal H_\lambda\subset L^2(\de B_1)$ and 
                \begin{equation*}
                    c_{\lambda,m}:=\fint_{\de B_1}w\psi_{\lambda,m}.
                \end{equation*}
                In particular, if $\sup_{0<r<1/4}r^{-\sigma}\|w(r\cdot)\|_{L^2(B_1)}$ is finite for some $\sigma>0$, then $c_{\lambda,m}=0$ for all $\lambda<\sigma$.
            \end{enumerate}
        \end{proposition}

    \subsection{Elliptic estimates}\label{subsec:elliptic}
        We collect some local regularity properties of functions $f$ satisfying the ellipticity condition $f\lap f\geq 0$. On the one hand, if $w$ solves \eqref{eq:thinobstaclelinearized} then $w\lap w\equiv 0$. On the other hand we have the following. 
        \begin{remark}\label{rem:wlapw}
        If $u$ and $u'$ both solve \eqref{eq:thinobstacle}, then
        \[(u-u')\lap(u-u')\geq0.\]
        Indeed, $(u-u')\lap(u-u') = -u\Delta u' - u'\Delta u$. This is nonnegative since both $u$ and $u'$ are nonnegative on $\{x_n=0\}$, while $\Delta u$ and $\Delta u'$  are nonpositive and supported on $\{x_n=0\}$.
        \end{remark}
        \begin{lemma}\label{lem:ellipticestimates}
            Suppose $w\colon B_2\to\R$ satisfies $w\lap w\geq 0$. Then
            \begin{equation*}
                \sup_{B_{3/2}}|w|+\|\nabla w\|_{L^2(B_{3/2})}\leq C \|w\|_{L^2(B_2\setminus B_1)},
            \end{equation*}
            where $C$ depends only on $n$. Moreover, $r \mapsto H_0(r,w)$ is nondecreasing.
        \end{lemma}
        \begin{proof}
            Fix a smooth cut-off function $\xi$ satisfying $\mathcal X_{B_{3/2}}\leq \xi\leq \mathcal X_{B_{2}}$. 
            Since $w\lap w\geq 0$, 
            we have 
            \[
            \int_{B_{3/2}} |\nabla w|^2 \le \int_{B_{2}} |\nabla w|^2\xi \le  \frac 12 \int_{B_{2}} \Delta (w^2)\xi
            = \frac 12 \int_{B_{2}} w^2 \Delta\xi  \le C \int_{B_2\setminus B_1} w^2.
            \]
            
            The $L^\infty$ bound and the fact that $H_0$ is nondecreasing follow easily using that $w^2$ is subharmonic and the mean value property.
        \end{proof}
        \begin{remark}
            Let $e\in\{x_n=0\}$ and $u$ a solution of \eqref{eq:thinobstacle}. Applying \Cref{lem:ellipticestimates} to $w:=u(te+\cdot)-u$ and sending $t\downarrow 0$ one gets the following ``tangential'' $H^2$ estimate
            \begin{equation}\label{eq:H2tangentialestimates}
                \|\nabla \de_e u\|_{L^2(B_1)}\leq C\|u\|_{L^2(B_2\setminus B_1)},
            \end{equation}
            where $C$ depends on $n$.
        \end{remark}
        As an application of this remark, we can compare the frequency at different points in the contact set.
        \begin{lemma}\label{lem:frequencycomparison}
            Let $u$ solve \eqref{eq:thinobstacle} with $0$ and $x_\circ$ in $\Lambda(u)\cap B_{2r}$ for some $r\in(0,1/2)$. Then we have
            \begin{equation*}
                \left|\frac{\phi(r,u)}{\phi(r,u(x_\circ +\cdot))}-1\right|\leq C^{\phi(4r,u)}\frac{|x_\circ|}{r},
            \end{equation*}
            where $C$ depends on only on $n$.
        \end{lemma}
        \begin{proof}
            By scaling, it is enough to prove that if $u$ solves \eqref{eq:thinobstacle} in $B_R$, for some $R\geq 4$, then
            \begin{equation}\label{eq:closenessbyfreqscaled}
                \left|\frac{\phi(1,u)}{\phi(1,u(z +\cdot))}-1\right|\leq|z| C(n)^{\phi(4,u)}\text{ for all }z\in \Lambda(u)\cap B_2.
            \end{equation}
            Fix $z$ and set $v:=u(z+\cdot)$, by \cref{eq:H2tangentialestimates}, \Cref{lem:Hlambdau} and \Cref{lem:ellipticestimates} we have
            \begin{align*}
                |D(1,u)-D(1,v)|&\leq \|\nabla(u-v)\|_{L^2(B_1)}\|\nabla(u+v)\|_{L^2(B_1)}\\
                &\leq C(n)|z|\|\nabla\de_{\frac{z}{|z|}} u\|_{L^2(B_{1+|z|})}(\|u\|_{L^2(B_{3/2})}+\|v\|_{L^2(B_{3/2})})\\
                &\leq C(n)|z|\|u\|_{L^2(B_4)}^2\leq C(n) |z| 4^{\phi(4,u)} H(1,u),
            \end{align*}
            and similarly we get
            \begin{align*}
                |H(1,u)-H(1,v)|&\leq C(n) |z| 4^{\phi(4,u)} H(1,u).
            \end{align*}
            
            On the other hand we have 
            \[
            \begin{split}
            \left|\frac{\phi(1,u)}{\phi(1,v)} -1\right| 
            &= \left|\frac{D(1,u)H(v,1)-D(1,v)H(u,1)}{D(1,v)H(u,1)}\right| \\
            &= \left|\frac{\big(D(1,u)-D(1,v)\big)H(v,1) - D(1,v)(H(1,u)-H(1,v)}{D(1,v)H(u,1)}\right| 
            \\
            &\le \frac{\big|D(1,u)-D(1,v)\big|}{H(1,u) \phi(1,v)} + \frac{|H(1,u)-H(1,v)|}{H(u,1)} 
            \end{split}
            \]
            So, using $\phi(1,v)\geq \phi(0+,v)\geq 1$ we obtain \eqref{eq:closenessbyfreqscaled}.
        \end{proof}
        We use also this general result
        \begin{lemma}\label{lem:holderdecay}
            There are $C(n)>0,\alpha(n)\in(0,1)$ such that if $f\colon B_1\to\R$, $f\lap f\geq 0$ in $B_1$ and
            \begin{equation*}
                B_1\cap \{x_{n-1} < -\delta, x_n=0\} \subset \{f=0\}
            \end{equation*}
            for some $\delta\in (0,1/20)$, then
            \begin{equation}\label{eq:holderdecay}
                \sup_{\mathcal{N}_{2\delta}\cap B_{1/2}}|f|\leq C\delta^\alpha \|f\|_{L^2(B_1\setminus B_{1/2})},
            \end{equation}
            where $\mathcal{N}_\delta:=\{x_n^2+(x_{n-1})_+^2\leq \delta^2\}$. %Furthermore, if we also have
            %\begin{equation*}
             %   \supp \lap f\subset \{x_{n-1}<\delta,x_n=0\},
            %\end{equation*}
            %then $\osc(f,B_{\delta}(z))\leq C\delta^\alpha\|f\|_{L^2(B_1\setminus B_{1/2})},$ for all $z\in B_{1/4}$.
        \end{lemma}
        \begin{proof}
            By assumption $f^2$ is sub-harmonic, so the mean value inequality gives
            \begin{equation}\label{eq:L2toLinfty}
                \sup_{B_{\rho/2}(z)}f^2\leq C(n) \|f(z+\rho\cdot)\|_{L^2(B_1\setminus B_{1/2})}^2\quad \text{ for all }B_{2\rho}(z)\subset B_1.
            \end{equation}
            Now take $z\in \mathcal{N}_\delta\cap\{x_n=0\} \cap B_{1/2}$, a smooth cut-off $\mathcal{X}_{B_{1/2}}\leq \xi\leq \mathcal{X}_{B_1}$ and some $\rho\in[4\delta,1/4]$. If we set $\tilde f:=f(z+\rho\cdot)$ the choice of $\rho$ gives that (independently from $z$) $\tilde f$ vanish in $B_1\cap\{x_{n-1}<-2/3,x_n=0\}$. Hence, these $\tilde f$ satisfy a Poincar\'{e} inequality $\int_{B_1\setminus B_{1/2}}\tilde{f}^2\leq C(n)\int_{B_1\setminus B_{1/2}}|\nabla \tilde f|^2$. Now compute
            \begin{align}
                \int_{B_1}|\nabla \tilde f|^2 |x|^{2-n}\xi^2(x)\, dx & \leq \frac 1 2 \int_{B_1}\lap(\tilde f^2)|x|^{2-n}\xi^2 \notag\\ \notag
                &=\frac 1 2 \int_{B_1} \tilde f^2\big(\underbrace{-c_n\delta_0\xi^2}_{\leq 0} + 2\nabla(|x|^{2-n})\cdot\nabla(\xi^2) +|x|^{2-n}\lap(\xi^2)\big)\\ \label{wwbeitohow}
                &\leq C(n)\int_{B_1\setminus B_{1/2}}\tilde{f}^2\leq C(n)\int_{B_1\setminus B_{1/2}}|\nabla \tilde f|^2
            \end{align}
            where we used that the derivatives of $\xi$ are supported in the annulus, and the aforementioned Poincar\'{e} inequality. So if we define the scale-invariant functional
            $$
            I(z,\rho):=\int_{B_\rho(z)}|\nabla f(y)|^2|y-z|^{2-n}\, dy=\int_{B_1}|\nabla \tilde f(x)|^2|x|^{2-n}\, dx
            $$
            by the {\em hole-filling trick}\footnote{Namely, note that as a consequence of \eqref{wwbeitohow}  we have shown $I(z,\rho/2) \le C\int_{B_1\setminus B_{1/2}}|\nabla \tilde f|^2  \le C(I(z,\rho)-I(z,\rho/2))$.} we find $I(z,\rho/2)\leq \tfrac{C}{1+C}I(z,\rho)$ so iterating and using that $\rho\mapsto I(z,\rho)$ is increasing:
            \begin{equation*}
                I(z,4\delta)\leq C(n) \delta^{2\alpha} I(z,1/4)\quad \forall z\in \mathcal N_\delta\cap \{x_n=0\}\cap B_{1/2}.
            \end{equation*}
            Now using \eqref{eq:L2toLinfty}  with $\rho:=4\delta$ and Poincare' in the annulus we find
            $$
            \sup_{B_{2\delta}(z)}|f|\leq C \|\tilde f\|_{L^2(B_1\setminus B_{1/2})}\leq C\|\nabla \tilde f\|_{L^2(B_1\setminus B_{1/2})}\leq C I(z,4\delta)^{\frac12} \leq C\delta^\alpha I(z,1/4)^{\frac12}
            $$
            to conclude we observe that our computation (now with $\rho=1/4$) also gives
            $$
            I(z,1/4)\leq C\int_{B_1}\tilde f^2\leq C \sup_{B_{1/4}(z)}f^2\leq C \sup_{B_{3/4}}f^2\leq C\|f\|_{L^2(B_1\setminus B_{1/2})}^2,
            $$
            which finally gives \eqref{eq:holderdecay}, since $z$ was arbitrary. %Notice that we do not need the weighted Poincar\'{e} inequality, because we apply it in the annulus where the weight $|x|^{2-n}\sim_n 1$.
            
            % We prove the second part of the statement. For $k\geq 0$ we split the domain in the sets
            % $$
            % A_0:=\mathcal{N}_\delta\cap B_{1/4};\ A_k:= B_{1/4}\cap (\mathcal{N}_{4^{k+1}\delta}\setminus \mathcal{N}_{4^{k}\delta});\  \tilde A_k:= B_{1/4}\cap (\mathcal{N}_{2\cdot 4^{k+1}\delta}\setminus \mathcal{N}_{4^{k}\delta/2}).
            % $$
            % Given $z\in B_{1/4}$, if $z\in A_0$ then $B_\delta(z)\subset \mathcal N_{2\delta}$ so we can apply the first part of this Lemma. Hence we assume $z\in A_k$ for some $k\geq 1$ (this forces $4^k\delta<1$ otherwise $A_k=\emptyset$), so that we have $B_\delta(z)\subset \tilde A_k$. Since $f$ is harmonic in $\tilde A_k$, there is $z^*\in\de \tilde A_k$ such that
            % $$
            % \sup_{\tilde A_k}|\nabla f|=|\nabla f(z^*)|\leq \fint_{B_{R/2}(z^*)}|\nabla f|\lesssim_n\frac 1 R \sup_{B_{R}(z^*)} |f|
            % $$
            % provided $B_R(z^*)\subset B_1\setminus\mathcal N_\delta$. The choice $R:=4^k\delta/100$ also ensures $B_R(z^*)\subset B_{1/2}\cap \mathcal N_{4^{k+2}\delta}$ and gives us
            % $$
            % \osc(f,B_\delta(z))\leq\delta\sup_{\tilde A_k}|\nabla f|\leq \frac {\delta}R\sup_{B_{1/2}\cap \mathcal N_{4^{k+2}\delta}}|f|\leq C(n) 4^{k(\alpha-1)}\delta^\alpha \|f\|_{L^2(B_1\setminus B_{1/2})},
            % $$
            % where in the last step we applied the first part of the statement at larger scales, namely with $\delta\leftarrow R$, this can be done as $\delta<R<1/20$. Since the constant on the right hand side is uniformly bounded for $k\geq 0$, we are done.
        \end{proof}
    \subsection{Estimates on the difference between {\it u} and a 2D solution}\label{subsec:estimatesdifference}
        In this subsection we give some estimates on $w:=u-\psi_\lambda$, with $\lambda\in\{\frac32,\frac72,\ldots\}\cup\{1,3,5,\ldots\}$. They will be crucial in \Cref{sec:frequency}. We remark that the proofs for the case $\lambda$ half-odd-integer present the difficulties of both the cases $\lambda$ even and $\lambda$ odd.
        
        The explicit relationship between $\eta$ and $\delta$ in the following Lemma will be crucial. 
        \begin{lemma}\label{lem:barrier}
            Let $u$ solve \eqref{eq:thinobstacle}, $\lambda\in\{\tfrac{3}{2},\tfrac{7}{2},\tfrac{11}{2},\ldots\}$ and $\eta,\delta \in (0,1)$. If for some $\tau>0$ we have $\sup_{B_{1}}|u-\tau\psi_\lambda|\leq \eta$ and $C\eta \leq \tau\delta^\lambda$, then
            \begin{align}
	            \label{eq:barrierinclusionhard} B_{1-\delta} \cap \{x_{n-1}<-\delta,x_n=0\}&\subset \{u=0\}, \\
	            \label{eq:barrierinclusioneasy} B_{1}\cap \{x_{n-1}>\delta,x_n=0\}&\subset \{u>0\}.
            \end{align}
            The large constant $C$ depends only on $n$ and $\lambda$.
        \end{lemma}
        \begin{proof}
            Let us deal with \eqref{eq:barrierinclusionhard}. By inspection of the (explicit) function $\psi_\lambda$ we see that there are positive constants $a,B$ (depending on $\lambda$ and $n$) such that
            \begin{equation}\label{eq:psilambdabarrierbounds}
                \psi_\lambda(-e_{n-1}+x)\leq -a|x_n|+B|x|^2\quad \text{for all } x\in \overline B_{1/2}.
            \end{equation}
            Fix some $z\in \{x_n=0\}\cap B_{1-\delta}$ with $z_{n-1}\leq-\delta$, we are going to prove that if $C(n,\lambda)\eta\leq \tau\delta^\lambda$ then $u(z)=0$. Using $\lambda$-homogeneity and translation invariance of $\psi_\lambda$ we get for all $y\in \de B_\rho$ with $\rho\leq \delta/2$ that 
            \begin{align*}
                u(z+y)&\leq \tau\psi_\lambda(z_{n-1}e_{n-1}+y)+\eta=\tau|z_{n-1}|^{\lambda}\psi_\lambda(-e_{n-1}+y/|z_{n-1}|)+\eta\\
                &\leq \tau |z_{n-1}|^{\lambda-2}\left(-a|z_{n-1}||y_n|+B\rho^2+\frac{\eta}{\tau |z_{n-1}|^{\lambda-2}}\right)=:(*),
            \end{align*}
            where we used $y/|z_{n-1}|\in B_{1/2}$, since $|z_{n-1}|\geq \delta \geq 2\rho$. If we take $\eta< B\rho^2\tau \delta^{\lambda-2}$, recalling $|y|=\rho$ and $|z_{n-1}|\geq \delta$, we have
            \begin{align*}
                (*)&<\tau|z_{n-1}|^{\lambda-2}\left(-a\delta|y_n| + 2B\rho^2 \right)\\
                &=\tau|z_{n-1}|^{\lambda-2}\left(2By_n^2-a\delta|y_n| + 2B|y'|^2+2By_{n-1}^2 \right)\quad \mbox{on } \partial B_\rho,
            \end{align*}
            where $y'= (y_1, \dots, y_{n-2})$.
            Now for $\rho:= \delta/C(a,B,n)$ (this gives the claimed dependence of $\eta$ from $\delta$) we have $-a\delta|y_n|\leq -4nB y_n^2$. We proved that for all constants $c> 0:$
            \begin{equation}\label{eq:barrieraux}
                u(z+y)<2\tau B|z_{n-1}|^{\lambda-2}\left(-2ny_n^2+|y'|^2+y_{n-1}^2+c\right):=\phi_{z,c}(y)\quad\text{ for all }y\in\de B_\rho.
            \end{equation}
            Notice that $\lap \phi_{z,c}<0$ everywhere. Let $c^*$ be the smallest value for which $u(z+\cdot)\leq \phi_{z,c}$ in $\overline B_\rho$: if $c^*=0$ then $u(z)\leq \phi_{z,0}(0)=0$ and \eqref{eq:barrierinclusionhard} follows. Lets us prove that $c^*>0$ is impossible.
            
            If $c^*>0$, then there is $x^*\in \overline B_\rho(z)$ such that $\phi_{z,c^*}$ touches from above $u(z+\cdot)$ at $x^*$. In fact, $x^*\in B_\rho$ because $u(z+\cdot)$ and $\phi_{z,c^*}$ never touch on $\de B_\rho(z)$, by \eqref{eq:barrieraux}. Also, $u(z+\cdot)$ cannot be harmonic around $x^*$, because we would have $0>\lap \phi_{z,c}(x^*) \geq \lap u(z+x^*)=0$. This proves $x^*_n=0$ and $u(z+x^*)=0$, an so we find a contradiction, namely
            $$
            0=u(z+x^*)=\phi_{z,c^*}(x^*)=2\tau B|z_{n-1}|^{\lambda-2}(|{x^*}'|^2+{x_n^*}^2+c^*)>0.
            $$

            Now we prove \eqref{eq:barrierinclusioneasy}. If $\eta\leq \tau\delta^\lambda /C$ then for all $z\in B_{1}\cap \{u=0,x_n=0\}$ we have
            $$
            \tau (z_{n-1})_+^\lambda = |(u-\tau\psi_\lambda)(z)|\leq \eta\leq \tau\delta^\lambda /C,
            $$
            so $z_{n-1}<\delta$ as we can take $C>1$. Since $u\geq 0$ on the thin space, this proves \eqref{eq:barrierinclusioneasy}.
        \end{proof}
        We have an analogous Lemma for odd frequencies.
        \begin{lemma}\label{lem:barrierodd}
            Let $u$ solve \eqref{eq:thinobstacle}, $\lambda\in\{1,3,5,\ldots\}$ and $\eta,\delta \in (0,1)$. If for some $\tau>0$ we have $\sup_{B_{1}}|u-\tau\psi_\lambda|\leq \eta$ and $C\eta \leq \tau\delta^\lambda$, then
            \begin{align}\label{eq:barrierinclusionodd}
                B_{1-\delta} \cap \{|x_{n-1}|>\delta,x_n=0\}&\subset \{u=0\}.
            \end{align}
            The large constant $C$ depends only on $n$ and $\lambda$.
        \end{lemma}
        \begin{proof}
            We can repeat word by word the proof of \eqref{eq:barrierinclusionhard} in \Cref{lem:barrier} above. In fact, in that argument, we only used that $\psi_\lambda$ satisfies \eqref{eq:psilambdabarrierbounds}, that it is $\lambda$-homogeneous and that it is translation invariant along the directions $1,2,\dots,{n-2}$. All these properties are true for also $\lambda$ odd.
        \end{proof}
        We can estimate now the laplacian of $u-\tau\psi_\lambda$.
        \begin{lemma}\label{lem:laplacianL1}
            Let $u$ solve \eqref{eq:thinobstacle}, $\lambda\in\{\tfrac{3}{2},\tfrac{7}{2},\frac{11}{2},\ldots\}\cup\{1,3,5,\ldots\}$, $\tau\geq 0$ and $\delta \in (0,1/10)$. 
            If $w:=u-\tau\psi_\lambda$ then
            \begin{equation*}
                \frac 1 C\int_{\mathcal{C}_\delta\cap B_{1}}|\lap w|\leq \tau \delta^\lambda +\|w\|_{L^2(B_2\setminus B_{1})},
            \end{equation*}
            where $\mathcal{C}_\delta = \{x_{n-1}^2+x_n^2\leq \delta^2\}$ and $C=C(n,\lambda)>0$.
        \end{lemma}
        \begin{proof}[Proof in the case $\lambda\in \frac 3 2 + 2\N$]
            We first spell out $\lap w$. By \cref{eq:laplacianpsilambda}, for some constant $c(n,\lambda)>0$ we have on the thin space
            \begin{equation}\label{eq:laplaciantable}
                \lap w(x)=
                \begin{cases}
                    c \tau (x_{n-1})_-^{\lambda-1}\, \delta_0(x_n) \geq 0& \text{ in }\{u>0\},\\
                    \lap u \leq 0& \text{ in }\{u=0, x_{n-1}\geq 0\},\\
                    \lap u +c \tau  (x_{n-1})_-^{\lambda-1}\, \delta_0(x_n) & \text{ in }\{u=0, x_{n-1}\leq 0\}.
                \end{cases}
            \end{equation}
            in particular we find that, as measures,
            \begin{equation*}
                (\lap w)_+\leq -\tau\lap \psi_\lambda= c \tau (x_{n-1})_-^{\lambda-1}\, \delta_0(x_n).
            \end{equation*}
            This shows that if we pick any $z\in\{x_{n-1}=x_n=0\}\cap B_1$ we have, using the notation $x=(x',x_{n-1},x_n)$,
            \begin{equation*}
                \int_{Q_{2\delta}(z)}(\lap w)_+\leq c \tau \int_{[-2\delta,2\delta]^{n-2}}dx'\int_{-2\delta}^{2\delta} dx_{n-1} (x_{n-1})_-^{\lambda-1}\leq C(n,\lambda) \tau \delta^\lambda\delta^{n-2},
            \end{equation*}
            where $Q_\delta(z)$ is the cube of center $z$ and half-side $\delta$. Take now a smooth cut-off $\mathcal X_{B_{1}}\leq \xi\leq \mathcal X_{B_{2}}$ and setting $\tilde w:=w(z+\cdot)$ and $\xi_{1/\delta}=\xi(\cdot/\delta)$ we have, integrating by parts the laplacian:
            \begin{align*}
                \int_{Q_\delta(z)}|\lap w|=\int_{Q_\delta}|\lap \tilde w|& \leq -\int_{B_1}\lap \tilde w \xi_{1/\delta} +2\int_{B_1} (\lap \tilde w)_+\xi_{1/\delta}\\
                &\leq -\delta^{n-2}\int_{B_1} \lap \tilde w_{\delta } \xi + 2\int_{Q_{2\delta}(z)}(\lap w)_+\\
                &\leq C(n)\delta^{n-2}\sup_{B_\delta(z)}|w|+C(n,\lambda)\tau\delta^{n-2} \delta^\lambda\\
                &\leq \delta^{n-2} C(n,\lambda)\|w\|_{L^2(B_2\setminus B_1)}+C(n,\lambda)\tau\delta^{n-2} \delta^\lambda.
            \end{align*}
            Summing these as $z$ vary on a $(n-2)$-dimensional lattice of length $\delta$ we find the sought estimate.
        \end{proof}
        \begin{proof}[Proof in the case $\lambda\in 1+2\N$]
            The proof is identical to the case $\lambda\in \frac 3 2 +2\N$, since $-\lap\psi_\lambda=c|x_{n-1}|^{\lambda-1}\delta_0(x_n)$ and $|x_{n-1}|^{\lambda-1}$ has the same size of $(x_{n-1})_+^{\lambda-1}$ in $\mathcal C_\delta$.
        \end{proof}
        We combine the  previous Lemmas into a ``nonlinear'' bound.
        \begin{lemma}\label{lem:estimatewlapw}
            Let $u$ solve \eqref{eq:thinobstacle}, $\lambda\in\{\tfrac{3}{2},\tfrac{7}{2},\tfrac{11}{2},\ldots\}\cup\{1,3,5,\ldots\}$, $\tau\geq 0$ and $S\in O(n-1)$. If $w:=u-\tau\psi_\lambda\circ S$ then
            \begin{equation}\label{eq:boundwlapw}
                \int_{B_1}|w\lap w|\leq C \tau^{-\alpha/\lambda}\|w\|_{L^2(B_2\setminus B_1)}^{2+\alpha/\lambda},
            \end{equation}
            where $C=C(n,\lambda)$ and $\alpha=\alpha(n)>0$.
        \end{lemma}
        \begin{proof}[Proof in the case $\lambda\in \frac 3 2 + 2\N$]
            Scaling $w$ and rotating the coordinates, we can prove the statement in the case $\tau=1$ and $S=\bm 1$. Since $w\lap w\geq 0$ by Remark \ref{rem:wlapw} we can apply \Cref{lem:ellipticestimates,lem:barrier} to $w$ and find
            \begin{align*}
	            B_1 \cap \{x_{n-1}<-\delta,x_n=0\}&\subset \{w=0\}\text{ and }\\
	            B_{1}\cap \{x_{n-1}>\delta,x_n=0\}&\subset \{\lap w=0\},
            \end{align*}
            for $\delta:=C(n,\lambda)(\|w\|_{L^2(B_2\setminus B_1)})^{1/\lambda}$. In particular for $\|w\|_{L^2(B_2\setminus B_1)}\leq \eta(n,\lambda)$ small we have $\delta<1/20$, hence applying \Cref{lem:holderdecay,lem:laplacianL1} to $w$ we find
            \begin{align*}
                \int_{B_1}|w\lap w|& = \int_{B_1\cap \mathcal C_\delta}|w\lap w|\leq \sup_{B_1\cap \mathcal C_\delta}|w|\int_{B_1\cap \mathcal C_\delta}|\lap w|\\
                & \leq C(n,\lambda) \delta^\alpha\|w\|_{L^2(B_2\setminus B_1)}\left( \delta^\lambda+\|w\|_{L^2(B_2\setminus B_1)}\right),
            \end{align*}
            and \eqref{eq:boundwlapw} follows by the definition of $\delta$. We conclude noticing that, when $\|w\|_{L^2(B_2\setminus B_1)}$ is large, then \eqref{eq:boundwlapw} is trivial since using a smooth cut-off $\mathcal X_{B_{1}}\leq \xi\leq \mathcal X_{B_{2}}$ we have
            \[
                \int_{B_1}|w\lap w| \le \int_{B_2}w\lap w \xi \le \frac 1 2 \int_{B_2}\lap (w^2) \xi = \frac 1 2 \int_{B_2} w^2 \lap\xi
                \le C\int_{B_2\setminus B_1} w^2,
            \]
            which implies that  \eqref{eq:boundwlapw} holds (for some $C= C(n,\lambda)$) if $\|w\|_{L^2(B_2\setminus B_1)}\geq \eta(n,\lambda)$.
        \end{proof}
        \begin{proof}[Proof in the case $\lambda\in 1 + 2\N$]
            The proof is identical to the case $\lambda$ half-odd-integer after we notice that, by \Cref{lem:ellipticestimates,lem:barrierodd}, also in this case we have
            \begin{equation*}
                B_1\cap \supp w\lap w\subset \mathcal C_\delta,\text{ for } \delta:=C(n,\lambda)(\|w\|_{L^2(B_2\setminus B_1)})^{1/\lambda}.\qedhere
            \end{equation*}
        \end{proof}
\section{Blow-ups are 2D up to dimension reduction}\label{sec:2Dblowups}
    \subsection{}For convenience we gather in this section some facts that are already known, and we state them in a convenient fashion.
    
    For each $\lambda\geq 0$, $u$ solving \eqref{eq:thinobstacle}, $\eps>0$ and $r\in(0,1)$ consider the following statement:
        \begin{equation}\label{propertyP}
            \begin{split}
            \mathcal P_\eps^\lambda(r,u):=&\text{``there exists a linear space $L\subset\{x_n=0\}$ such that}\\
            &\quad \text{$\dim L\leq n-3$ and $\Lambda_\lambda(u)\cap B_{r/2} \subset L+B_{\eps r}$''.}
            \end{split}
        \end{equation}
        We define now the set of bad points of frequency $\lambda$:
        \begin{equation}\label{eq:deflLambdastar}
            \Lambda^*_\lambda(u):=\bigcap_{\eps>0}\left\{x_\circ\in \Lambda_\lambda(u) : \limsup_k\frac{\#\{0\leq j\leq k : \mathcal P_\eps^\lambda(2^{-j},u(x_\circ+\cdot))\text{ is true}\}}{k}=1\right\}.
        \end{equation}
        We also define the set of points where the blow up cannot be 2D (cf. \Cref{lem:2Dsolutionsclassification}),
        \begin{equation}
            \Lambda_{\text{\upshape{other}}}(u):=\left\{x_\circ\in\Lambda(u) : \phi(0+,u(x_\circ+\cdot))\not\in \big(\tfrac32+2\N\big)\cup\N\right\}.
        \end{equation}
        In this section we prove three facts:
        \begin{itemize}
            \item if $x_\circ\in\Lambda_\lambda(u)$ and $\mathcal P_\eps^\lambda(r_k,u(x_\circ+\cdot))$ fails along a sequence of scales $r_k\downarrow 0$, then there exist a 2D blow-up of $u$ at $x_\circ$;
            \item points at which $\mathcal P_\eps^\lambda(r,u)$ is ``asymptotically always'' true (i.e., points in $\Lambda^*_\lambda(u)$), have Hausdorff dimension at most $n-3$.
            \item $\Lambda_{\text{\upshape{other}}}(u)$ has Hausdorff dimension at most $n-3$.
        \end{itemize}
    \subsection{2D blow-ups}
        For each $r\in(0,2)$ and $\lambda\in\{\frac32,\frac72,\ldots\}\cup\{1,2,3,\ldots\}$, we define the quantity
        \begin{equation}\label{eq:defofAfunctional}
            A_\lambda(r,u):=\min_{\tau\in\R,S\in O(n-1)}\|(u-\tau\psi_\lambda\circ S)(r\cdot)\|_{L^2(\de B_{1})},
        \end{equation}
        which is increasing in $r$ by \Cref{lem:Hlambdau}. We remark that $A(r,u)=A(1,u_r)$.
        
        \begin{lemma}\label{lem:closeness by frequency}
            Let $\lambda\in\{\frac 3 2,\frac 7 2 ,\frac {11}2,\ldots\}\cup\{1,2,3,4,\ldots\}$. For all $\eps,\eta\in(0,1/2)$ there exist $\delta_\circ>0$ and $M_\circ\geq 1$ with the following property. If $u$ solves \eqref{eq:thinobstacle} with $0\in\Lambda_\lambda(u)$, and  for some $r\in(0,1/M_\circ)$ we have that
            \begin{equation}\label{eq:closeness by freq assumption}
                \phi(M_\circ r,u)\leq \lambda+\delta_\circ \text{ and }\mathcal P_\eps^\lambda(r,u)\text{ fails,}
            \end{equation}
            then $A(r,u)\leq \eta \|u_{r}\|_{L^2(\de B_{1})}$.
        \end{lemma}
        \begin{proof}
            Assume by contradiction there exist some positive $\eta_\circ,\eps_\circ$ such that the statement does not hold. Rescaling, this means that there are functions $v_k:=u_k(r_k\cdot)/\|u_k(r_k\cdot)\|_{L^2(\de B_1)}$ which satisfy
            \begin{equation*}
                \lambda=\phi(0+,v_k)\leq \phi(M_k,v_k)\leq \lambda+\delta_k,\ A(1,v_k) >  \eta_\circ\text{ and }\mathcal P_{\eps_\circ}^\lambda(1,v_k)\text{ fails,}
            \end{equation*}
            for some sequences $M_k\uparrow\infty$ and $\delta_k\downarrow 0$. If $z_k\in\Lambda_\lambda(v_k)$, then applying \Cref{lem:frequencycomparison} to $v_k,z_k$ and $M_k/4$ we find that, for each $\rho\in(0,M_k/4)$, it holds
            \begin{align}\label{eq:closeness by freq local}
                \lambda\leq \phi(\rho,v_k(z_k+\cdot))&\leq \phi(M_k/4,v_k(z_k+\cdot)) \\
                & \leq \phi(M_k/4,v_k)\Big(1+\frac{C(n,\lambda)}{M_k}\Big) \leq (\lambda+\delta_k)\Big(1+\frac{C(n,\lambda)}{M_k}\Big).\notag
            \end{align}
            
            By the $C^{1,\alpha}$ regularity of solutions of \eqref{eq:thinobstacle} (see for instance \cite{AC}) and Ascoli-Arzel\`a, we obtain that  $v_k\to v_*$ in $C^1_\loc(\{x_n\ge 0\})$ (up to subsequence), where $v_*$ is a $\lambda$-homogeneous solution of \eqref{eq:thinobstacle}. Now, since $\mathcal P_{\eps_\circ}^\lambda(1,v_k)$ fails for all $k$, we can find $(n-2)$ linearly independent points $z_1,\ldots,z_{n-2}$ such that
            \begin{equation*}
                z_j\in \Lambda(v_*)\cap \{x_n=0\} \cap (\overline B_1\setminus B_{\eps_\circ})\text{ for all }j\in\{1,\ldots,n-2\}.
            \end{equation*}
            Passing to the limit \eqref{eq:closeness by freq local} one gets that in fact $z_j\in\Lambda_\lambda(v_*)$ for all $j$'s. Now, since 
            \[\lambda = \phi(0+,v_*(z_j+ \,\cdot\,)) \le \phi(\infty,v_*(z_j+ \,\cdot\,)) = \phi(\infty,v_*) = \lambda\] 
            we obtain that $v_*$ is $\lambda$-homogeneous about the $n-2$ linearly independent points $z_j$ (and also about $0$ as shown above). Hence, $v_*$ must be translation invariant along $(n-2)$ directions, so using \Cref{lem:2Dsolutionsclassification} and the fact that $\|v_*\|_{L^2(\de B_1)}=1$, we get that up to a rotation $v_*=\psi_\lambda$, contradicting $A(1,v_*)\geq \eta_\circ>0$.
        \end{proof}
        \begin{lemma}\label{lem:lambdaother}
            Let $u$ solve \eqref{eq:thinobstacle} with $0\in\Lambda_{\text{\upshape{other}}}(u)$ and $\eps>0$. Then there exists $\rho_0>0$ with the following property. For all $r\in(0,\rho_0)$ there exists a linear space $L\subset \{x_n=0\}$ such that $\dim L\leq n-3$ and 
            \begin{equation*}
                \big\{x\in B_r : |\phi(0+,u)-\phi(0+,u(x + \cdot))|\leq \rho_0\big\}\subset L + B_{\eps r}.
            \end{equation*}
        \end{lemma}
        \begin{proof}[Sketch of the proof.] We argue by contradiction exactly as in \Cref{lem:closeness by frequency}. We end up with a nontrivial solution of \cref{eq:thinobstacle} $v_*$ which is translation invariant along $n-2$ directions and $\lambda$ homogeneous for $\lambda:=\phi(0+,u)$. Since nontrivial 2D solutions of \eqref{eq:thinobstacle} must have homogeneities in $\big(\frac32+2\N\big) \cup \N$ this is a contradiction. See \cite{FS} for more details.
        \end{proof}
%        \begin{remark}
%            We can repeat the proof \Cref{lem:closeness by frequency} for any $\lambda\geq 0$, and obtain in the limit a $\lambda$-homogeneous solution of \eqref{eq:thinobstacle} in 2D. Since such solutions can exist only for $\lambda\in\{2,3,\ldots\}\cup\{\frac 32,\frac72,\ldots\}$ and are classified, an analogous version of \Cref{lem:closeness by frequency} holds for $\lambda\in\{2,3,\ldots\}$ and $\psi_\lambda$ as defined in \eqref{eq:psilambdadef}. Assume now $\lambda\not\in\{2,3,\ldots\}\cup\{\frac 32,\frac72,\ldots\}$. Then the proof of \Cref{lem:closeness by frequency} shows that if $\lambda\leq\phi(0+,u)\leq\phi(1/2,u)\leq \lambda+\delta_\circ$ then $\mathcal P_\eps^\lambda(r,u)$ holds true for all $r\in(0,r_\circ)$. A Geometric Measure Theory argument then gives $\dim_{\mathcal{H}}\Lambda_\lambda(u)\leq n-3$ for such $\lambda$'s. This was first proved in ...%TODO:cite
%        \end{remark}
    \subsection{Dimension reduction}
        We prove the following proposition.
        \begin{proposition}\label{prop:dimensionreduction}
             Let $\lambda\in\{\frac 3 2,\frac 7 2 ,\frac {11}2,\ldots\}\cup\{1,2,3,\ldots\}$ and $u$ solve \eqref{eq:thinobstacle}, then $$\dim_{\mathcal H}\Lambda^*_\lambda(u)\leq n-3.$$
        \end{proposition}
        The proof is a direct consequence of a more general result, we introduce some notation. For all $E\subset\R^n$, $x\in\R^n$ and $\ell\in\N$ define the number
        \begin{align*}
            N^\eps(E,x,\ell):=\#\big\{j\in\{0,\ldots,\ell\}&: \text{exists a linear space $L$, with $\dim L\leq n-3$,} \\
            &\text{ such that }E\cap B_{2^{-j}}(x)\subset x+L+B_{\eps2^{-j}}\big\}.
        \end{align*}
        Then we have the following lemma
        \begin{lemma}[{\cite[Lemma 12.6]{FRS2}}]\label{lem:dimredauxiliary}
            Let $E\subset \R^n$. Given $\eps,\omega>0$ small assume that
            \begin{equation*}
                \limsup_{\ell\to\infty}\frac{N^\eps(E,x,\ell)}{\ell}\geq 1-\omega\text{ for all }x\in E.
            \end{equation*}
            Then $\dim_{\mathcal H}E\leq n-3+\alpha$, where $\alpha(n,\eps,\omega)\to 0$ as $(\eps,\omega)\to(0,0)$.
        \end{lemma}
        \begin{proof}[Proof of \Cref{prop:dimensionreduction}]
            By definition of $\Lambda^*_\lambda(u)$ we have, for all $\eps>0$,
            \begin{equation*}
                \limsup_{\ell\to\infty}\frac{N^\eps(\Lambda_\lambda^*(u),x_\circ,\ell)}{\ell} = 1\text{ for all }x_\circ\in \Lambda^*_\lambda(u).
            \end{equation*}
            Hence \Cref{lem:dimredauxiliary} applied to $E:=\Lambda^*_\lambda(u)$ gives $\dim_{\mathcal H}\Lambda^*_\lambda(u)\leq n-3+\alpha$, where $\alpha>0$ can be made arbitrarily small (by choosing $\eps$ accordingly small). This concludes the proof.
        \end{proof}
        We also record the following result
        \begin{proposition}\label{prop:dimredotherfreq}
            Let $u$ solve \eqref{eq:thinobstacle}, then $\dim_{\mathcal H}\Lambda_{\text{\upshape{other}}}(u)\leq n-3.$
        \end{proposition}
        The proof relies on a semi-classical GMT Lemma
        \begin{lemma}[{\cite[Proposition 7.3]{FRS1}}]\label{lem:dimredauxiliaryotherfreq}
            Let $E\subset \R^n$, $0\leq m\leq n$ and $f\colon E\to \R$.
            Assume that, for any $\eps>0$ and $x\in E$ there exists $\varrho= \varrho(x,\eps)>0$ such that, for all $r\in (0, \varrho)$, we have
            $$
            E\cap {B_r(x)}\cap f^{-1}([f(x)-\varrho,f(x)+\varrho])\subset L+B_{\eps r},
            $$
            for some $m$-dimensional plane $L$ passing through $x$ (possibly depending on $r$). Then $\dim_\mathcal{H}(E)\leq m$.
        \end{lemma}
        \begin{proof}[Proof of \Cref{prop:dimredotherfreq}]
            We apply \Cref{lem:dimredauxiliaryotherfreq} to $E:=\Lambda_{\text{\upshape{other}}}(u)$, $f(x):=\phi(0+,u(x+\cdot))$ and $m:=n-3$. Assumptions of \Cref{lem:dimredauxiliaryotherfreq} are satisfied because of \Cref{lem:lambdaother}, namely we can take $\varrho:=\rho_0$.
        \end{proof}
\section{Accelerated decay for $\lambda$ half-odd-integer}\label{sec:accelerateddecay}
        
        \subsection{}In this section we prove the following result (recall that the set $\Lambda^*$ appearing in its statement was defined in \cref{eq:deflLambdastar}).
        \begin{proposition}\label{prop:acceleratedecay}
            Let $\lambda\in\{\frac 3 2,\frac 7 2 ,\frac {11}2,\ldots\}$ and $u$ solve \eqref{eq:thinobstacle}. For each $x_\circ \in B_1\cap ( \Lambda_\lambda(u)\setminus \Lambda^*_\lambda(u))$ there exist $\beta,r_\circ>0$ and $S\in O(n-1)$ (all depending on $x_\circ$ and $u$) such that
            \begin{equation}\label{eq:accellerateddecay}
                \sup_{B_r}|u(x_\circ+\cdot)-q|\leq (r/r_\circ)^{\lambda+\beta}\|u(r_\circ\cdot)\|_{L^2(\de B_1)}\text{ for all }r\in(0,r_\circ),
            \end{equation}
            where $q:=H_\lambda(0+,u(x_\circ+\cdot))\psi_\lambda\circ S$. Furthermore, $H_\lambda(0+,u(x_\circ+\cdot))>0$.
        \end{proposition}
       
    \subsection{A Monneau-type monotonicity formula}\label{subsec:monneau}
        % \begin{lemma}(This Lemma is not necessary anymore, since we can use directly \Cref{lem:estimatewlapw})
        %     For all $\delta\in(0,1)$ there exists $\varphi\in C^{1}(\overline B_2)$ such that
        %     \begin{equation*}
        %         \begin{cases}
        %                 \lap\varphi\leq 0 & \text{ in }B_2,\\
        %                 \varphi = 0 & \text{ on }\de B_2,\\
        %                 |\nabla\varphi|\leq C &\text{ on }\de B_2,\\
        %                 \varphi\geq -\log\delta/C & \text{ in }B_1\cap \C_\delta,
        %         \end{cases}
        %     \end{equation*}
        %     where $\C_\delta:=\{x_{n-1}^2+x_n^2\leq \delta^2\}$ and the constant $C$ depends only on $n$.
        % \end{lemma}
        We prove the following monotonicity formula, which is inspired by a similar computation in \cite[Section 10]{FRS2}.
        \begin{proposition}\label{prop:weakmonneau}
            Let $w:=u-\tau\psi_\lambda\circ S$, with $u$ solving \eqref{eq:thinobstacle}, $\lambda\in\{\frac 3 2,\frac 7 2 ,\frac {11}2,\ldots\}$, $\tau>0$ and $S\in O(n-1)$. Assume further that $0\in\Lambda_\lambda(u)$. Then for all $\mu\in(0,\lambda)$ the function
            \begin{equation*}
                r\mapsto  \max_{s\in[1,2]}H_\mu(sr,w)=:\widetilde H_{\mu}(r,w)
            \end{equation*}
            is increasing in the interval $[\varrho,1/2]$ where
            \begin{equation}\label{eq:rhoweakmonneausize}
                \varrho := C_1\sup_{B_{1/2}}|w/\tau|^{1/\lambda},
            \end{equation}
            for $C_1=C_1(n,\lambda,\mu)$.
        \end{proposition}
        \begin{proof}
            Without loss of generality we can prove the statement under the assumptions $\tau=1$ and $S=\bm 1$. Let $\varrho$ be as in \eqref{eq:rhoweakmonneausize} with $C_1(n,\lambda,\mu)$ to be chosen.
            
            At point $r_\circ\in(\varrho,1/2)$ there are only two cases:
            \begin{itemize}
                \item[(a)] $\widetilde H_{\mu}(r_\circ,w) > H_\mu(r_\circ,w)$;
                \item[(b)] $\widetilde H_{\mu}(r_\circ,w) = H_\mu(r_\circ,w)$. 
            \end{itemize}
            In case (a), by continuity there must exist $\theta > 1$ very close to $1$, such that
            $$
            \max_{s\in [1,2]} s^{-2\mu}r_\circ^{-2\mu} H_0(sr_\circ ,w) \leq \max_{s\in [\theta,2]} s^{-2\mu}r_\circ^{-2\mu} H_0(sr_\circ,w).
            $$
            This allows us to compute for all  $t\in[1,\theta]$
            \begin{align*}
                \widetilde H_{\mu}(t r_\circ,w)&=\max_{s\in [1,2]} (s tr_\circ)^{-2\mu} H_0(st r_\circ,w)\\
                &=\max_{\bar s\in [\rho,2\rho]} (\bar sr_\circ)^{-2\mu} H_0(\bar sr_\circ,w)\\
                &\geq \max_{s\in [1,2]} (sr_\circ)^{-2\mu} H_0(sr_\circ ,w)=\widetilde H_{\mu}(r_\circ,w),
            \end{align*}
            and this proves monotonicity in case (a).
            
            Let us address case (b), where by assumption we know
            \begin{align}\label{eq:freedoubling}
                H_0(sr_\circ,w)\leq s^{2\mu}H_0(r_\circ,w)\quad \text{ for all }s\in[1,2].
            \end{align}
            It is sufficient to show that the (right) derivative of $r\mapsto H_\mu(r,w)$ at $r=r_\circ$ is non-negative, since for $\eps>0$
            \begin{equation*}
                \frac{\widetilde H_{\mu}(r_\circ+\eps,w)-\widetilde H_{\mu}(r_\circ,w)}{\eps}\geq \frac{H_\mu(r_\circ+\eps,w)-H_\mu(r_\circ,w)}{\eps},
            \end{equation*}
            where we used that $\widetilde H_{\mu}(r,w)\geq H_\mu(r,w)$, with equality if $r=r_\circ$ (thanks to (b)). We claim that 
            \begin{equation}\label{eq:Hmuderivative}
                \frac{d}{dr}H_\mu(r,w) \geq \frac{2r^{-2\mu}}{r} \left((\lambda-\mu)\int_{\de B_1} w_r^2 - \int_{B_1} w_r\lap w_r\right)\text{ for all }r\in(0,1),
            \end{equation}
            and we show how to conclude from here. We apply \Cref{lem:estimatewlapw} to $w_{r_\circ}$ and that as a consequence of \eqref{eq:freedoubling} we have $\|w_{r_\circ}\|_{L^2( B_2\setminus B_1)}^2 \le 2^{\mu+n/2} \|w_{r_\circ}\|_{L^2( \partial B_1)}^2$, to obtain:
            \begin{align*}
                \int_{B_1}|w_{r_\circ}\lap w_{r_\circ}| & \leq C(n,\lambda) \left(\frac{\|w_{r_\circ}\|_{L^2( B_2\setminus B_1)}}{\tau r_\circ^\lambda}\right)^{\alpha/\lambda}\|w_{r_\circ}\|_{L^2( B_2\setminus B_1)}^2\\
                & \leq C(n,\lambda) \left(\frac{\|w_{r_\circ}\|_{L^2( 
                \de B_1)}}{\tau r_\circ^\lambda}\right)^{\alpha/\lambda}\|w_{r_\circ}\|_{L^2(\de B_1)}^2.
            \end{align*}
            Now, since $\varrho\le r_\circ\le 1/2$, where $\varrho$ is and since $w^2$ is subharmonic (and hence $H(r,w)^{1/2} = \|w_r\|_{L^2(\de B_1)}$  is nondecreasing in $r$)   we have
            $$
            \left(\frac{\|w_{r_\circ}\|_{L^2( 
                \de B_1)}}{\tau r_\circ^\lambda}\right)^{\alpha/\lambda}\leq \left(\frac{\|w\|_{L^2( 
                \de B_{1/2})}}{C_1\sup_{B_{1/2}}|w|}\right)^{\alpha/\lambda}\leq C_1^{-\alpha/\lambda}.
            $$
            Hence choosing $C_1(n,\lambda,\mu)$ large we can ensure
            $$
            \int_{B_1}|w_{r_\circ}\lap w_{r_\circ}|\leq \frac{\lambda-\mu}{2}H_0(r_\circ,w),
            $$
            and inserting this inequality in \eqref{eq:Hmuderivative}, we find $H'_\mu(r_\circ,w)>0$.
            
            It remains to prove \cref{eq:Hmuderivative}. A direct computation gives 
            \begin{equation}\label{eq:derivativeformula}
                \frac{d}{dr}H_\mu(r,w)=\frac{2r^{-2\mu}}{r} \left(\int_{B_1}|\nabla w_r|^2 + \int_{B_1}w_r\lap w_r -\mu\int_{\de B_1} w_r^2\right)
            \end{equation}
            We set $q:=\tau\psi_\lambda\circ S$ and manipulate $D(r,w)$ in the following way (using in the first inequality below that $\phi(r,u)\geq\lambda$ and $\phi(r,q)\equiv \lambda$):
            \begin{align*}
                \int_{B_1}|\nabla w_r|^2&=\int_{B_1}|\nabla u_r|^2+\int_{B_1}|\nabla q_r|^2-2\int_{B_1}\nabla u_r\cdot \nabla q_r\\
                &\geq \lambda \int_{\de B_1} (u^2_r+q^2_r) -2 \int_{\de B_1} u_r \underbrace{x\cdot \nabla q_r}_{=\lambda q_r} +2\int_{B_1} u_r\lap q_r\\
                &=\lambda \int_{\de B_1} w_r^2 +2\int_{B_1} u_r\lap q_r\\
                &\geq \lambda \int_{\de B_1} w_r^2 + 2\int_{B_1} u_r\lap q_r+2\int_{B_1} \underbrace{q_r\lap u_r}_{\leq 0}\\
                &=\lambda \int_{\de B_1} w_r^2 -2 \int_{B_1} w_r\lap w_r,
            \end{align*}
            inserting this in \eqref{eq:derivativeformula} we have precisely \eqref{eq:Hmuderivative}. 
        \end{proof}
    \subsection{A dichotomy result}

        We need a preliminary lemma concerning solutions of the linearized problem \eqref{eq:thinobstaclelinearized}. In the following lemma and in the sequel it will be useful to consider the following ``diameter-type'' quantity (which will be used in connection to  $\mathcal P_\eps^\lambda(r,u)$). Given $X\subset \R^n$ and and integer $0\le m\le n$ let us define 
        \begin{equation}\label{diamn-3}
            {\rm diam}_m(X) : = 
                \min \{d\ge 0 \ :\ X\subset L+B_d,  \mbox{ for some $L\in {\rm Gr}( m, \R^n)$}\},
         \end{equation}       
        where ${\rm Gr}(m,\R^n)$ is the set for all $m$-dimensional linear subspaces of $\R^n$,

        \begin{lemma}\label{lem:auxiliary}
            For all $\lambda\in\{\frac 3 2,\frac 7 2 ,\frac {11}2,\ldots\}$, $\eps>0$ small, $M\geq 1 $ large there exists $\sigma>0$ such that the following holds. If
            \begin{enumerate}
                \item $w$ solves the linearized problem \eqref{eq:thinobstaclelinearized} in $B_1$ with $\|w\|_{L^2(\de B_1)}\leq 1$; 
                \item $\| w_*(\rho\cdot)\|_{L^2(\de B_1)}\leq M \rho^{\lambda-1/3}$ for all $\rho\in(0,1/4)$;
                \item $\int_{\de B_1}w\psi_\lambda=\int_{\de B_1}wx_i\psi_{\lambda-1}=0$ for all $i\in\{1,
                \ldots,n-2\}$;
                \item there exist $x_1,\ldots,x_{n-2}$ in $B_{1/2}\cap\{x_{n-1}=x_n=0\}$ with 
                \[{\rm diam}_{n-3}(\{x_1, \dots, x_{n-2}\})\ge \eps,\] 
                such that $\rho\|\de_i w(x_j+\rho\cdot)\|_{L^2(B_1)}\leq M \rho^{\lambda-1/3}$ for all $\rho\in(0,1/4)$ and for all $i,j\in\{1,\ldots,n-2\}$;
            \end{enumerate}
            then $\|w(\cdot/2)\|_{L^2(\de B_1)}\leq (1/2)^{\lambda+\sigma}$.
        \end{lemma}
        \begin{proof}
            We argue by contradiction and find a sequence $w_k$ solving \eqref{eq:thinobstaclelinearized} in $B_1$ such that (1),(2),(3) and (4) hold for suitable $M_\circ,\eps_\circ$, but $\|w_k(\cdot/2)\|_{L^2(\de B_1)}> (1/2)^{\lambda+\sigma_k}$ for some sequence $\sigma_k\downarrow 0$. We orthogonally split $w_k=f_k+g_k$ with $f_k\in\mathcal H_{\leq\lambda}$ and $g_k\in\mathcal H_{\geq \lambda+1}$. Thanks to (2) we actually have $f_k\in\mathcal H_\lambda$. This gives
            \begin{equation*}
                (1/2)^{2\lambda+2\sigma_k}<\|w_k(\cdot/2)\|_{L^2(\de B_1)}^2\leq(1/2)^{2\lambda}\|f_k\|_{L^2(\de B_1)}^2+(1/2)^{2\lambda+2}\|g_k\|_{L^2(\de B_1)}^2,
            \end{equation*}
            which, combined with $\|f_k\|_{L^2(\de B_1)}^2+\|g_k\|_{L^2(\de B_1)}^2=\|w_k\|^2_{L^2(\de B_1)}\leq 1$, gives
            \begin{equation}\label{eq:auxiliary}
                \|f_k\|_{L^2(\de B_1)}\to 1\text{ and } \|g_k\|_{L^2(\de B_1)}^2\to 0 \text{ as }\sigma_k\downarrow 0.
            \end{equation}
            We recall that \Cref{lem:ellipticestimates} holds for solutions of \eqref{eq:thinobstaclelinearized}, so for each $\rho<1$ we have
            \begin{equation*}
                \rho\|\nabla w_k(\rho\cdot)\|_{L^2(B_1)}\leq C_\rho \|w_k\|_{L^2(\de B_1)}.
            \end{equation*}
            Hence up to taking sub-sequences $w_k\weak w_\infty$ in $H^1_\loc(B_1)$ and in $L^2(\de B_1)$, so that $\int_{\de B_1}w_\infty\psi_\lambda=\int_{\de B_1}w_\infty x_i\psi_{\lambda-1}=0$ for all $i\in\{1,                \ldots,n-2\}$. \Cref{eq:auxiliary} entails $w_\infty\in \mathcal H_\lambda$ and by compactness of the trace $\|w_\infty(\cdot/2)\|_{L^2(\de B_1)}\geq (1/2)^{\lambda}$. Let us fix some $j\in\{1,\ldots,n-2\}$ and notice that $\de_jw_\infty\in\mathcal H_{\lambda-1}$. Item (4) gives, up to taking sub-sequences, that we have $(n-2)$ linearly independent points $x_1,\ldots,x_{n-2}$  such that the function $\de_jw_\infty(x_i+\cdot)$ (which belongs to $\mathcal H_{\leq \lambda-1}$, since $\de_jw_\infty\in\mathcal H_{\lambda-1}$) vanish with order $\lambda-1-1/3$ at the origin, so in fact $\de_jw_\infty(x_i+\cdot)\in\mathcal H_{\lambda-1}$. This gives that, for each $j$, $\de_j w_\infty$ is translation-invariant along $n-2$ directions. By \Cref{lem:2Dsolutionsclassification} we get $\de_jw_\infty (x)=\beta_j \psi_{\lambda-1}(x_{n-1},x_n)$ for some $\beta\in\R^{n-2}$. Integrating and using that $w_\infty\in\mathcal H_\lambda$ we find
            $$
                w_\infty(x',x_{n-1},x_n)=\alpha\psi_\lambda(x_{n-1},x_n)+(\beta\cdot x')\psi_{\lambda-1}(x_{n-1},x_n)\text{ for all}x\in\R^n,
            $$
            for some $\alpha\in\R$. This, together with $(4)$, gives $w_\infty\equiv0$, a contradiction as $\|w_\infty(\cdot/2)\|_{L^2(\de B_1)}\geq(1/2)^{\lambda}$.
        \end{proof}
        
        The following key result establishes an alternative, which holds at every scale $r$: either property $\mathcal P_\eps^\lambda(r,u)$ --- defined in \eqref{propertyP} --- holds true, or else the functional  $A =A_\lambda$ --- defined in \eqref{eq:defofAfunctional} --- satisfies an ``enhanced decay'' between scales $r$ and $r/2$.
        \begin{proposition}\label{prop:dichotomy}
            Let $\lambda\in\{\tfrac 3 2,\tfrac 7 2 , \tfrac{11}{2},\ldots\}$. For all $\eps,\gamma\in(0,1/3)$, there exists $\eta_\circ>0$ with the following property. If $u$ solves \eqref{eq:thinobstacle} with $0\in\Lambda_\lambda(u)$, and  for some $r\in(0,1)$ we have that
            \begin{equation*}
                A(r,u)\leq \eta_\circ \|u_r\|_{L^2(\de B_1)},
            \end{equation*}
            then the following decay holds:
            \begin{itemize}
                \item[(a)] if $\mathcal P_\eps^\lambda(r,u)$ holds, then $A(r/2,u)\leq(1/2)^{\lambda-\gamma}A(r,u)$;
                \item[(b)] if $\mathcal P_\eps^\lambda(r,u)$ fails, then $A(r/2,u)\leq(1/2)^{\lambda+\sigma}A(r,u)$.
            \end{itemize}
            Where $\sigma>0$ is half the exponent in \Cref{lem:auxiliary} and depends only on $n,\lambda,\eps$.  
        \end{proposition}
        \begin{proof}
            The statement is scale-invariant so we can prove it at $r=1$. We start proving case (a), and then strengthen the argument for case (b).
            
            By contradiction, assume that for some $\eps_\circ,\gamma\in(0,1/3)$ we can find $\eta_k\downarrow 0$ and $u_k$ solving $\eqref{eq:thinobstacle}$ with $0\in\Lambda_\lambda(u_k)$ and $\|u_k\|_{L^2(\de B_1)}=1$, such that  \begin{equation}\label{eq:dichotomycontradictiona}
                A(1,u_k)\leq \eta_k \quad \text{ and }\quad A(1/2,u_k) >  (1/2)^{\lambda-\gamma}A(1,u_k).
            \end{equation}
            We take the optimal $\tau_k\in\R$ and $S_k\in O(n-1)$ in such a way that if we set
            \begin{equation}
                w_k:=u_k-\tau_k\psi_\lambda\circ S_k\quad \text{ and }\quad \widetilde w_k:=\frac{w_k}{A(1,u_k)}
            \end{equation}
            then $\|\widetilde w_k\|_{L^2(\de B_1)}=1$. By compactness we can assume that coordinates were chosen in such a way that $S_k\to\bm 1$. Also, 
            since $\|u_k-\tau_k\psi_\lambda\circ~S_k\|_{L^2(\de B_1)} = A(1,u_k)\to 0$
            ---by \eqref{eq:dichotomycontradictiona}--- while
            $\|u_k\|_{L^2(\de B_1)}=1$, 
            we find that it must be $\tau_k\to 1$. We claim that $\widetilde w_k\weak w_*$ weakly in $H^1_\loc(B_1)$ where $w_*$ has the following properties
            \begin{enumerate}
                \item $w_*$ solves the linearized problem \cref{eq:thinobstaclelinearized} and $\| w_*\|_{L^2(\de B_1)}\leq 1$;
                \item $\| w_*(\rho\cdot)\|_{L^2(\de B_1)}\leq C(n,\lambda) \rho^{\lambda-1/3}$ for all $\rho\in(0,1/4)$.
            \end{enumerate}
            We also claim that, as a consequence of the second inequality in \ref{eq:dichotomycontradictiona}, we have
            \begin{equation}\label{wntoiwhwoih}
                (1/2)^{\lambda-\gamma}\leq \| w_*(\cdot/2)\|_{L^2(\de B_1)}.
            \end{equation}

            Since properties (1) and (2) readily implies that $w_*\in\mathcal{H}_{\geq \lambda}$ we find a contradiction with \eqref{wntoiwhwoih}, namely
            \begin{equation*}
                (1/2)^{2\lambda-2\gamma}\leq \| w_*(\cdot/2)\|_{L^2(\de B_1)}^2=\sum_{\lambda'\geq \lambda} \|w_{*\lambda'}(\cdot/2)\|^2_{L^2(\de B_1)}\leq (1/2)^{2\lambda},
            \end{equation*}
            impossible since $\gamma>0$. Here we used $\|w_*\|^2_{L^2(\de B_1)}\leq 1$ and denoted with $w_{*\lambda'}$ the $L^2(\de B_1)$-projection of $w_*$ onto $\mathcal H_{\lambda'}$.
            
            \textit{Proof of properties (1),(2) and \eqref{wntoiwhwoih}}. Since $w_k\lap w_k\geq 0$ (by Remark \ref{rem:wlapw}), boundedness of $\widetilde w_k$ in $H^1_\loc(B_1)\cap L^\infty_\loc(B_1)$ follows by \Cref{lem:ellipticestimates} and the fact that $\| w_k\|_{L^2(\de B_1)}\leq 1$. Now using \eqref{eq:dichotomycontradictiona} and the $L^\infty_\loc(B_1)$ bound we find for each fixed $\rho<1$ that
            \begin{align}\label{eq:dichotomyfreeboudnarylocation}
	            B_\rho \cap \{x_{n-1}<-{\delta_{k,\rho}},x_n=0\}&\subset \{w_k=0\}\text{ and }\\
	            B_{\rho}\cap \{x_{n-1}> \delta_{k,\rho},x_n=0\}&\subset \{\lap w_k=0\},
            \end{align}
            for $\delta_{k,\rho}:=C(n,\lambda,\rho)\eta_k^{1/\lambda}$. So, in the limit $k\uparrow\infty$, $w_*$ is harmonic in $B_1\setminus \mathcal N_0$ and has trace zero on $\mathcal N_0$. On the other hand by compactness of the trace operator and monotonicity of $H_0(\cdot,w_k)$ we have
            \begin{align*}
                \| w_*(\rho\cdot)\|_{L^2(\de B_1)}=\lim_k\| \widetilde w_k(\rho\cdot)\|_{L^2(\de B_1)}\leq \limsup_k\| \widetilde w_k\|_{L^2(\de B_1)}\leq 1,
            \end{align*}
            letting $\rho\uparrow 1$ we find (1). In order to prove (2), we apply \Cref{prop:weakmonneau} to $w_k$ with $\mu:=\lambda-1/3$ and find
            \begin{equation*}
                \rho^{-2(\lambda-1/3)}H(\rho,w_k)\leq \widetilde H_{\mu}(\rho,w_k)\leq \widetilde H_{\mu}(1/4,w_k)\leq C(\lambda)\| w_k\|_{L^2(\de B_1)}^2
            \end{equation*}
            for all $\rho\in(\varrho_k,1/4)$ where $\varrho_k\leq C(n,\lambda)\eta_k^{1/\lambda}$. Dividing by $A(1,u_k)^2$ and sending $k\uparrow \infty$ we find
            \begin{equation*}
                \rho^{-2(\lambda-1/3)}H(\rho,w_*)\leq C(\lambda)\text{ for all }\rho\in(0,1/4),
            \end{equation*}
            which is (2). We prove \eqref{wntoiwhwoih} using \eqref{eq:dichotomycontradictiona}, the compactness of the trace and the definition of the $A$ functional
            \begin{equation*}
                (1/2)^{\lambda-\gamma}\leq\lim_k \frac{A(1/2,u_k)}{A(1,u_k)}\leq \lim_k\|\widetilde w_k(\cdot/2)\|_{L^2(\de B_1)}=\|w_*(\cdot/2)\|_{L^2(\de B_1)}.
            \end{equation*}
            
            We can now turn to the proof of (b). In this case the contradiction assumption is stronger
            \begin{equation}\label{eq:dichotomycontradictionb}
                A(1,u_k)\leq\eta_k,\mathcal P_{\eps_\circ}^\lambda(1,u_k)\text{ fails, and }A(1/2,u_k) >  (1/2)^{\lambda+\sigma}A(1,u_k).
            \end{equation}
            We claim that $\widetilde w_k\weak w_*$ weakly in $H^1_\loc(B_1)$ where $w_*$ has properties (1) and (2) above and also
            \begin{itemize}
                \item[(3)] $\int_{\de B_1}w_*\psi_\lambda=\int_{\de B_1}w_*x_i\psi_{\lambda-1}=0$ for all $i\in\{1,
                \ldots,n-2\}$;
                \item[(4)] there exist $M=M(n,\lambda)$ and $x_1,\ldots,x_{n-2}$ in $B_{1/2}\cap\{x_{n-1}=x_n=0\}$ with \[{\rm diam}_{n-3}(\{x_1, \dots, x_{n-2}\})\ge \eps,\] such that $\rho\|\de_i w_*(x_j+\rho\cdot)\|_{L^2(B_1)}\leq M \rho^{\lambda-1/3}$ for all $\rho\in(0,1/4)$ and for all $i,j\in\{1,\ldots,n-2\}$.
            \end{itemize}
            Properties (1),(2),(3) and (4) allow us to apply \Cref{lem:auxiliary} to $w_*$, finding $$\|w_*(\cdot/2)\|_{L^2(\de B_1)}\leq(1/2)^{\lambda+2\sigma},$$ which contradicts assumption \eqref{eq:dichotomycontradictionb}.
            
            \textit{Proof of property (3)}. We recall that, for each $k$, $\tau_k$ was optimal so
            \begin{equation*}
                0  = \left.\frac{d}{dt}\right|_{t=1}\int_{\de B_1}(u_k-t\tau_k\psi_\lambda\circ S_k )^2= -2\tau_k\int_{\de B_1}w_k\psi_\lambda\circ S_k,
            \end{equation*}
            and since coordinates are taken in such a way that $S_k\to\bm 1$, we have $\int_{\de B_1}w_*\psi_\lambda=0$ just dividing by $A(1,u_k)$ and sending $k\to\infty$ (recall that $\tau_k$ is close to 1). For the other equations in (3) we use the optimality of $S_k$, that is for each $\ell\in\{1,\ldots, n-2\}$ and $\theta\in\R$ we consider $R_\theta\in O(n-1)$ defined by
            $$
                R_\theta(x):=(x_1,\ldots,\cos\theta x_\ell +\sin\theta x_{n-1},\ldots,\cos\theta x_{n-1}-\sin\theta x_\ell,x_{n})\text{ for all }x\in\R^n.
            $$
            Then optimality gives
            \begin{align*}
                0 & = \left.\frac{d}{d\theta}\right|_{\theta=0}\int_{\de B_1}(u_k-\tau_k\psi_\lambda\circ S_k \circ R_\theta)^2\\
                & = 2\tau_k\int_{\de B_1}w_k\big(\de_\ell(\psi_\lambda\circ S_k)x_{n-1}-\de_{n-1}(\psi_\lambda\circ S_k)x_\ell\big),
            \end{align*}
            since $\de_\ell\psi_\lambda\equiv0$ and $\de_{n-1}\psi_\lambda=c_{\lambda}\psi_{\lambda - 1}$, we can conclude as before letting $k\uparrow \infty$.
            
            \textit{Proof of property (4).} Negating $\mathcal P_{\eps_\circ}^\lambda(1,u_k)$ we find $y_1^{k},\ldots, y_{n-2}^{k}\in B_{1/2}\setminus B_{\eps_\circ}$ such that
            \begin{equation*}
                {\rm diam}_{n-3}(\{y^k_1, \dots, y^k_{n-2}\})\ge \eps_\circ\text{ and } y^{k}_j\in\Lambda_\lambda(u_k)
            \end{equation*}
            for all $k\geq 0$ and $j\in\{1,\ldots,n-2\}$. We remark that thanks to \eqref{eq:dichotomyfreeboudnarylocation} we have that 
            \begin{equation}\label{eq:Lambdalocation}
                B_{1/2}\cap \Lambda_\lambda(u_k)\subset \mathcal{C}_{\delta_k}
            \end{equation}
            where $\delta_k\downarrow 0$. We used that points in the interior of the contact set cannot have half-integral homogeneity. Define the numbers $t_k:=(S_ky^{k})\cdot {e_{n-1}}$ and notice that $t_k\to 0$ because of \eqref{eq:Lambdalocation}.
            
            Let us set for brevity $q_k:=\tau_k \psi_\lambda\circ S_k$ and define for each $j\in\{1,\ldots,n-2\}$ and $k\geq0$ the functions
            $$
            W^j_k:=u(y^{k}_j+\cdot)-q_{k};\quad D^j_k:=q_{k}(y^{k}_j+\cdot)-q_{k},
            $$
            so that $w_k(y^{k}_j+\cdot)=W^j_k-D^j_k$. Notice that $W^j_k$ is of the form for which we can apply \Cref{prop:weakmonneau}, obtaining
            \begin{equation}\label{eq:monneauW}
                \rho^{-2(\lambda-1/3)}H(\rho,W^j_k)\leq \widetilde H_{\mu}(\rho,W^j_k)\leq \widetilde H_{\mu}(1/4,W^j_k)\text{ for all }\rho\in(\varrho_k',1/4)
            \end{equation}
            also here $\varrho_k'\downarrow 0$ because
            \begin{equation*}
                \varrho_k'^\lambda/C_1\leq \sup_{B_{1/4}} |W^j_k|\leq \sup_{B_{3/4}}|w_k|+\sup_{B_{1/4}}|D_k|\leq C(n,\lambda)(\eta_k+t_k)\downarrow 0.
            \end{equation*}
            From now on we drop the $j$ in our notation, as it will remain fixed.
            
             Assume also $t_k\neq 0$, otherwise $D_k\equiv 0$. We first observe that, if $D_k/\|D_k\|_{L^2(B_1)}\to D_\infty$, then $D_\infty=\pm \psi_{\lambda-1}$. To see this we compute
            \begin{align*}
                D_\infty=\lim_k\frac{D_k}{\|D_k\|}=\lim_k\frac{\tfrac{1}{t_k}(\psi_\lambda(S_k\cdot+t_k\bm{e_{n-1}})-\psi_\lambda(S_k\cdot))}{\tfrac{1}{t_k}\|\psi_\lambda(S_k\cdot+t_k\bm{e_{n-1}})-\psi_\lambda(S_k\cdot)\|}=\frac{\pm\de_{n-1}\psi_\lambda}{\|\de_{n-1}\psi_\lambda\|}=\pm\psi_{\lambda-1},
            \end{align*}
            so in particular $D_\infty\in\mathcal H_{\lambda-1}\setminus\{0\}$.
        
            We claim that there is some large constant $\overline M=\overline M(n,\lambda)$ such that, for all $j= 1,\dots,n-2$, we have 
            \begin{equation}\label{eq:dichotomycase1}
            \limsup_k \frac{\|D_k\|_{L^2(\de B_1)}}{A(1,u_k)}\leq \overline M.
            \end{equation}
            Indeed, assume by contradiction this is not the case, so that along some subsequence that we do not relabel we have, assuming $y^k\to y^*$ as $k\to \infty$,
            \begin{equation*}
            W_\infty:=\lim_k\frac{W_k}{\|D_k\|}=\lim_k  \frac{A(1,u_k)}{\|D_k\|}{\widetilde{w}_k(y^{k}+\cdot)} +\frac{D_k}{\|D_k\|}= c w_*(y^*+\cdot) \pm \psi_{\lambda-1}
            \end{equation*}
            for some $c\in[0,1/M]$. On the other hand if we divide \cref{eq:monneauW} by $\|D_k\|_{L^2(\de B_1)}^2$ and send $k\uparrow \infty$ along this subsequence we find
            \begin{equation}\label{eq:Mbound1}
                \rho^{-2(\lambda-1/3)} H(\rho,W_\infty)\leq \widetilde H_{\mu}(1/4,W_\infty)\leq C(n,\lambda)\Big(\sup_{B_{3/4}}|cw_*|+|\psi_{\lambda-1}|\Big)\leq C(n,\lambda),
            \end{equation}
            for all $\rho\in(0,1/4)$. On the other hand we have
            \begin{align}\label{eq:Mbound2}
                H(\rho,W_\infty)^{\frac12}\geq H(\rho,\psi_{\lambda-1})^{\frac12} -H(\rho,cw_*(y_*+\cdot))^{\frac12} \\
                \geq \rho^{\lambda-1}-\frac1M \sup_{B_{3/4}}|w_*|\geq \rho^{\lambda-1}-\frac{C(n,\lambda)}{M},\notag
            \end{align}
            so combining \eqref{eq:Mbound1} and \eqref{eq:Mbound2} one gets
            \begin{equation*}
                1\leq C_1(n,\lambda)\rho^{2/3}+\frac{C_2(n,\lambda)\rho^{1-\lambda}}{M}\text{ for all }\rho \in(0,1/4).
            \end{equation*}
            Clearly one can arrange the choice of $ \rho=\overline \rho(n,\lambda)$ and $ M=\overline M(n,\lambda)$ such that both terms in the right hand side are smaller than $1/2$, making this inequality impossible. Thus we proved \eqref{eq:dichotomycase1}.
            
            Now we repeat a similar reasoning considering (along some subsequence)
            \begin{equation*}
            \widetilde W_\infty:=\lim_k\frac{W_k}{A(1,u_k)}=\lim_k  {\widetilde{w}_k(y^{k}+\cdot)} +\frac{\|D_k\|}{A(1,u_k)}\frac{D_k}{\|D_k\|}= w_*(y^*+\cdot) + c \psi_{\lambda-1},
            \end{equation*}
            for some $|c|\leq \overline M$. 
            But, as before, we divide \cref{eq:monneauW} by $A(1,u_k)^2$ and send $k\uparrow \infty$ finding \footnote{The fact that in this equation the bound on the right hand side is universal will be important when we apply \Cref{lem:auxiliary}, in order to have that the exponent $\sigma$ only depends on $n,\lambda,\eps$. This is why we prove \eqref{eq:dichotomycase1} first.}
            \begin{equation*}
                \rho^{-2(\lambda-1/3)} H(\rho,\widetilde W_\infty)\leq \widetilde H_{\mu}(1/4,\widetilde W_\infty)\leq C(n,\lambda)\sup_{B_{3/4}}|w_*|+\overline M|\psi_{\lambda-1}|\leq M(n,\lambda),
            \end{equation*}
            for all $\rho\in(0,1/4)$.
            Since $y^*\in\{x_{n-1}=x_n=0\}$ we get that $\widetilde W_\infty$ solves \eqref{eq:thinobstaclelinearized} in $B_{3/4}$, so the Caccioppoli inequality for solutions of \eqref{eq:thinobstaclelinearized} gives, for every $\ell\in\{1,\ldots,n-2\}$, that
            \begin{equation*}
                \rho^2\|(\de_\ell w_*)(y^*+\rho\cdot)\|_{L^2( B_{1/2})}\leq \rho^2\|\nabla \widetilde W_\infty(\rho\cdot)\|^2_{L^2(B_{1/2})} \leq H(\rho,\widetilde W_\infty)\leq M\rho^{2(\lambda-1/3)}
            \end{equation*}
            for all $\rho\in(0,1/8)$. We crucially used that $\de_\ell \widetilde W_\infty=\de_\ell w_*(y^*+\cdot)$.
            
            Since we can repeat this reasoning for each $j$, we proved that $w_*$ satisfies (4), where $x_j = y^*_j$.
        \end{proof}
                Before proving \Cref{prop:acceleratedecay} we need the following observation
        \begin{lemma}\label{lem:sequencelemma}
            Let $\{a_j\}_{j\in\N}\subset \{0,1\}$ be a sequence such that
            \begin{equation*}
                \liminf_m\frac{a_0+\ldots+a_m}{m} > p > 0.
            \end{equation*}
            Then for all $m>0$ there exists $n\geq m$ such that
            \begin{equation}\label{eq:sequencelemma}
                \frac{a_n+a_{n+1}+\ldots+a_{n+j}}{j+1}\geq p\quad \text{ for all }j\geq 0.
            \end{equation}
        \end{lemma}
        \begin{proof}
            By contradiction there exists $M_\circ$ such that for all $n\geq M_\circ$ there is $j(n)\geq 0$ such that 
            \begin{equation}\label{eq:asd}
                \frac{a_n+a_{n+1}+\ldots+a_{n+j(n)}}{j(n)+1}< p.
            \end{equation}
            Notice that we can always refine the choice of $j(n)$ (enlarging it) in such a way that $a_{n+j(n)+1}=1$ and \eqref{eq:asd} still holds. Apply this for $\ell$ times starting at some $n_1>M_\circ$ with $a_{n_1}=1$ and then for $n_2:=n_1+j(n_1)+1$, ..., until $n_{\ell+1}:=n_\ell+j(n_{\ell})+1$. We have
            \begin{align*}
                \sum_{k=n_1}^{n_{\ell}+j(n_\ell)}a_k =\sum_{h =1}^\ell\sum_{k=n_\ell}^{n_{\ell}+j(n_\ell)}a_k<p \sum_{h =1}^\ell (j(n_\ell)+1),
            \end{align*}
            notice that the sum on the right is exactly the number of terms on the left, so looking along the subsequence $\{a_{n_\ell}\}$ we find
            $$
            \liminf_m\frac{a_0+\ldots+a_m}{m}\leq p,
            $$
            a contradiction.
        \end{proof}
        We can now prove
        \begin{proof}[Proof of \Cref{prop:acceleratedecay}]
            We can assume $x_\circ =0$. $0\not\in\Lambda^*_\lambda(u)$ means that there exists $\eps_\circ=\eps_\circ(x_\circ)>0$ such that, if we define for $j\geq 0$ the sequence 
            \begin{equation*}
                a_j:=\begin{cases}
                    0 & \text{ if $\mathcal{P}_{\eps_\circ}^\lambda(2^{-j},u)$ holds true,} \\
                    1 & \text{ if $\mathcal{P}_{\eps_\circ}^\lambda(2^{-j},u)$ fails,}
                \end{cases}
            \end{equation*}
            then 
            \begin{equation*}
                \liminf_m\frac{a_0+\ldots+a_m}{m} > p_\circ,
            \end{equation*}
            for some $p_\circ=p_\circ(x_\circ)>0$. Let $\sigma=\sigma(n,\lambda,\eps_\circ)$ be given by \Cref{prop:dichotomy}. We define the numbers $\gamma$ and $\beta$ by
            \begin{equation*}
                \gamma:=\frac{\sigma p_\circ}{2(1-p_\circ)},\quad \beta:=\frac{\sigma p_\circ}{2}
            \end{equation*}
            this choice ensures that
            \begin{equation}\label{eq:gammabetachoice}
                t\sigma -(1-t)\gamma\geq \beta\text{ for all }t\in[ p_\circ, 1].
            \end{equation}
            Apply \Cref{prop:dichotomy} with these $\eps_\circ$ and $\gamma$ to find $\eta_\circ$. Apply \Cref{lem:closeness by frequency} to such $\eta_\circ$ to find $M_\circ$ and $\delta_\circ$. We refine the choice of $\delta_\circ$ requiring $\delta_\circ<\beta/2$. Now \Cref{lem:sequencelemma} ensures we can find $n_\circ$ so large that $\phi(M_\circ 2^{-n_\circ},u)\leq \lambda+\delta_\circ$ (because $0\in\Lambda_\lambda(u)$), and $a_{n_\circ}=1$ and if we define
            $$
            N(k):=a_{n_\circ}+\ldots+a_{n_\circ+k-1}\text{ for all }k\geq 1,
            $$
            then $N(k)\geq kp_\circ$ for all $k\geq 1$.
            
            We set $r_\circ:=2^{-n_\circ}$ and prove by induction that 
            \[
            A(2^{-j}r_\circ,u)\le \big(\tfrac{1}{2}\big)^{(\lambda+\beta)j} \eta_\circ \|u_{r_\circ}\|_{L^2(\de B_1)}\text{ for all } j\geq 1.
            \]
            Note that since we have $N(k)\geq kp_\circ$, for all $k\geq 1$,  by our choice of $\gamma$ and $\beta$ is suffices to show that 
            \begin{equation}\label{eq:inductive}
                A(2^{-j}r_\circ,u)\leq \big(\tfrac{1}{2}\big)^{j\lambda+N(j)\sigma-(j-N(j))\gamma}\eta_\circ \|u_{r_\circ}\|_{L^2(\de B_1)}\text{ for all }j\geq 1,
            \end{equation}
            since by \eqref{eq:gammabetachoice} we have
            \[
            j\lambda+N(j)\sigma-(j-N(j))\gamma \ge j\lambda+ p_\circ j \sigma- (j-p_\circ j)\gamma = j (\lambda+\beta).
            \]
            
            We will reason by induction over $j$. The case $j=1$ is exactly \Cref{lem:closeness by frequency}. Let us prove that the validity of \eqref{eq:inductive}  for $j+1$ assuming its validity for $j$. Since $\phi(r_\circ,u)\leq \delta_\circ+\beta/2$ and $j\lambda+N(j)\sigma-(j-N(j))\gamma\geq j(\lambda+\beta)$ we find using \Cref{lem:Hlambdau} that
            \begin{align*}
                A(2^{-j}r_\circ,u)& \leq (1/2)^{j\lambda+N(j)\sigma-(j-N(j))\gamma}2^{j(\lambda+\beta/2)}\eta_\circ\|u_{r_\circ/2^j}\|_{L^2(\de B_1)} \\
                & \leq (1/2)^{j\beta/2}\eta_\circ \|u_{r_\circ/2^j}\|_{L^2(\de B_1)}<\eta_\circ \|u_{r_\circ/2^j}\|_{L^2(\de B_1)},
            \end{align*}
            so we can apply \Cref{prop:dichotomy} to $u_{r_\circ/2^j}$ and find the expected decay according to whether $\mathcal{P}_{\eps_\circ}(r_\circ/2^j,u)$ holds true or not.
            
            For each $j\geq 1$ we set $q_j:=\tau_j\psi_\lambda\circ S_j$ so that
            \begin{equation*}
                A(2^{-j}r_\circ,u)=\|(u_{r_\circ}-q_j)(2^{-j}\cdot)\|_{L^2(\de B_1)}.
            \end{equation*}            
            From \cref{eq:inductive} one finds that
            \begin{equation*}
                \|q_j-q_{j+1}\|_{L^2(\de B_1)}\leq C(\beta)2^{-j\beta}\eta_\circ\|u_{r_\circ}\|_{L^2(\de B_1)}\text{ and }\tau_j\to r_\circ^\lambda H_\lambda(0+,u).
            \end{equation*}
            Hence summing the geometric series and refining the choice of $\eta_\circ$, one gets \eqref{eq:accellerateddecay}.
            
            In order to complete the proof, assume by contradiction that $H_\lambda(0+,u)=0$. Then by \eqref{eq:accellerateddecay} we find $\|u(r\cdot)\|_{L^2(\de B_1)}\ll r^{\lambda+\beta}$, contradicting \Cref{lem:Hlambdau}.
        \end{proof}
\section{Accelerated decay implies frequency monotonicity}\label{sec:frequency}
    \subsection{}Throughout this section we fix a solution $u$ of \eqref{eq:thinobstacle} and assume $0\in\Lambda_\lambda(u)\setminus \Lambda^*_\lambda(u)$, for some $\lambda\in\{\tfrac 3 2,\tfrac 7 2, \tfrac{11}{2},\ldots\}\cup\{1,3,5,\ldots\}$. Thanks to the previous section we may suppose
    \begin{equation}\label{eq:extradecayassumption}
        \sup_{B_{2r}}|u-\psi_\lambda|\leq r^{\lambda+\beta}\text{ for all }r\in(0,r_\circ),
    \end{equation}
    for some $r_\circ,\beta\in(0,1)$. The goal of this section is showing
    \begin{proposition}\label{prop:decayimprovement}
        Let $u$ solve \eqref{eq:thinobstacle} and $\lambda\in \{\tfrac 3 2,\tfrac 7 2,\ldots\}\cup\{1,3,5,\ldots\}$. If \eqref{eq:extradecayassumption} holds for some positive $\beta$ and $r_\circ$, then
        \begin{equation}\label{eq:decayimprovement}
            \sup_{B_r}|u-\psi_\lambda|\leq Cr^{\lambda+1}\text{ for all }r\in(0,r_\circ),
        \end{equation}
        where $C=C(n,\lambda,r_\circ,\beta)$.
    \end{proposition}
    In order to prove this proposition we will exploit the almost monotonicity of the truncated frequency function of $w:=u-\psi_\lambda$.
    \subsection{Known formulas}
        For each $\gamma\geq 0$ define the functionals
        \begin{equation*}
            \phi_\gamma(r,w):=\frac{D(r,w)+\gamma r^{2\gamma}}{H_0(r,w)+r^{2\gamma}};\quad g_\gamma(r,w):=\frac{\|w_r\|^2_{L^2(B_2\setminus B_1)}}{H_0(r,w)+r^{2\gamma}}.
        \end{equation*}

        By a direct computation (see \cite[Lemma 2.3]{FRS1}) one finds 
        \begin{equation}
            \label{eq:formulaphigammaderivative}
            \frac{d}{dr} \phi_\gamma(r,f) \ge \frac{2}{r}  \frac{{\left( \int_{B_1} f_r\lap f_r\right)^2}}{ (H_0(r,f) + r^{2\gamma} )^2}+\frac 2 r \frac{\int_{B_1} \left(\phi_\gamma(r,f) f_r - x\cdot \nabla f_r\right)\lap f_r }{ H_0(r,f) + r^{2\gamma} },
        \end{equation}
        for all functions $f$ regular enough.
        We recall from \cite{FRS1} the following result which generalizes \Cref{lem:Hlambdau} to the ``truncated setting''. 
        \begin{lemma}[{\cite[Lemma 4.1]{FRS1}}]\label{lem:H ratio}
        Let $f\colon B_1\to\R$ such that $f\lap f\geq 0$.
        Assume that for some $\kappa >0$, $r_\circ>0$ and a constant $C_1>0$ we have
        \begin{equation}\label{eq:assumptionsHratio}
        \frac{d}{dr} \Big(\phi_\gamma(r,f) + C_1 r^{\kappa} \Big)\ge \frac 2 r  
        \frac{\left( \int_{B_1} f_r\lap f_r\right)^2}{\big(H_0(r,f) + r^{2\gamma}\big)^2} \quad \text{ for all } \,r\in (0,r_\circ).
        \end{equation}
        Then the following holds:
        \begin{enumerate}
        	\item Suppose that $\sup_{[0,R]}\phi_\gamma(\cdot, w)\leq \overline{\lambda}$, for some $R<r_\circ$. Then we have
        	$$
        	\frac{H(R,f)+R^{2\gamma}}{H(r,f)+r^{2\gamma}} \le  C_2  (R/r )^{2\overline{\lambda}+\delta}\quad \text{for all } r\in(0,R),
        	$$
        	for all $\delta>0$, where $C_2$ depends on $n,\gamma,\kappa, r_\circ, \overline\lambda, C_1, \delta$.
        	\item Let $\lambda_*\leq \phi_\gamma(0^+,f)$, which exists by \eqref{eq:assumptionsHratio}. Then
        	$$
        	\frac{H(R,f)+R^{2\gamma}}{H(r,f)+r^{2\gamma}} \ge  C_3  (R/r )^{2\lambda_*}\quad \text{ for all } 0<r<R<r_\circ,
        	$$
        	where $C_3>0$ depends on $n,\gamma,\kappa, r_\circ,\lambda_*,C_1$.
        \end{enumerate}
        \end{lemma}

    \subsection{Almost monotonicity}
        The key estimate is the following.
        \begin{lemma}\label{lem:keyestimateformonotonicity}
           Let $u$ solve \eqref{eq:thinobstacle} and $\lambda\in \{\tfrac 3 2,\tfrac 7 2,\tfrac{11}{2},\ldots\}\cup\{1,3,5,\ldots\}$. Assume \eqref{eq:extradecayassumption} holds for some positive $\beta$ and $r_\circ$ and set $w_r:=(u-\psi_\lambda)(r\cdot)$. Then for all $r\in(0,r_\circ)$ we have
            \begin{align}\label{eq:wrlapwr}
                \int_{B_1}|w_r\lap w_r|\leq Cr^{\alpha\beta/\lambda}\|w_r\|^2_{L^2(B_2\setminus B_1)}
            \end{align}
            and 
            \begin{equation}\label{eq:gradwrlapwr}
                -\int_{B_1} (x\cdot\nabla w_r) \lap w_r \geq  -Cr^{\alpha\beta/\lambda}\|w_r\|^2_{L^2(B_2\setminus B_1)},
            \end{equation}
            where $C$ depends on $n,\lambda$ and $\alpha(n)>0$ is the exponent in \Cref{lem:holderdecay}.
        \end{lemma}
        \begin{proof}[Proof in the case $\lambda\in \frac 32 + 2\N$]
            Equation \eqref{eq:wrlapwr} follows immediately joining Lemma \ref{lem:estimatewlapw} with assumption \eqref{eq:extradecayassumption}.
            
            We prove \eqref{eq:gradwrlapwr} splitting the integral in the regions $B_1\cap\{u_r>0\}$ and $B_1\cap\{u_r = 0\}$. Set for brevity $q:=\psi_\lambda$. First of all we notice that inside $\{u_r=0\}$ we have $w_r\equiv-q_r$, so $x\cdot \nabla w_r = \lambda w_r$, hence using \eqref{eq:wrlapwr} we find
            \begin{equation*}
                -\int_{B_1\cap\{u_r=0\}} (x\cdot\nabla w_r) \lap w_r = -\lambda \int_{B_1\cap\{u_r=0\}}w_r\lap w_r\geq -Cr^{\alpha\beta/\lambda}\|w_r\|^2_{L^2(B_2\setminus B_1)}.
            \end{equation*}
            We turn to the bound in the region $B_1\cap\{u_r > 0\}$. Recall that $x_n(\de_n w_r) \lap w_r \equiv 0$, because $\lap w_r$ is supported in the thin space. Furthermore, formula \eqref{eq:laplacianpsilambda} gives for some $c=c(n,\lambda)>0$
            \begin{equation*}
                \lap w_r (x)= c r^\lambda(x_{n-1})_-^{\lambda-1}\delta_0(x_n)\text{ for all } x\in B_1\cap\{u_r > 0\}.
            \end{equation*}
            Using this formula for any $j\in\{1,\ldots,n-1\}$ we get
            \begin{align*}
                -\int_{B_1\cap\{u_r>0\}} x_j\de_j w_r & \lap w_r  = \int_{B_1\cap\{u_r>0,x_{n-1}\leq 0\}} x_j\de_j u_r \lap q_r\\
                & = -c r^\lambda \int_{B_1\cap\{x_{n-1}\leq 0,x_n=0\}}\de_j(u_r)_+ ( x_j|x_{n-1}|^{\lambda-1})\,dx'\, dx_{n-1},
            \end{align*}
            where $x'\in\R^{n-2}$. Now if we take the domain $\Omega =B_1\cap\{x_{n-1}\leq 0,x_n=0\}$ and integrate by parts (inside $\{x_n=0\}$) we find ($\nu^\Omega$ is the outer normal):
            \begin{align*}
                -\frac{1}{c r^\lambda }\int_{B_1\cap\{u_r>0\}} x_j\de_j w_r \lap w_r & = - \int_{\de\Omega} (u_r)_+ x_j |x_{n-1}|^{\lambda-1}\nu^\Omega_j \, d\mathcal{H}^{n-2}\\
                & \qquad +\int_\Omega \underbrace{u_r\de_j(x_j|x_{n-1}|^{\lambda-1})}_{\geq 0\text{ for all }j}\\
                &\geq -\int_{\de\Omega} x_j \nu^\Omega_j (u_r)_+ |x_{n-1}|^{\lambda-1} d\mathcal{H}^{n-2}
            \end{align*}
            To estimate the last piece we notice that, for all $j$, the integrand vanishes on all pieces of $\de \Omega$ except on $\de B_1\cap\{u_r>0,x_{n-1}\leq 0,x_n=0\}$ where $\nu^\Omega(x)\equiv x$ and $(u_r)_+ \equiv w_r$. On the other hand by \Cref{lem:barrier} and assumption \eqref{eq:decayimprovement} we have \begin{equation*}
                 \de B_1\cap\{u_r>0,x_{n-1}\leq 0,x_n=0\}\subset \de B_1\cap \mathcal C_{\delta_r}
            \end{equation*}
            where
            \begin{equation*}
                \delta_r^\lambda\leq C(n,\lambda)\sup_{B_{3/2}}|r^{-\lambda}w_r| \le C(n,\lambda)r^\beta
            \end{equation*}
            Hence,  using that
            \begin{itemize}
                \item $\sup_{B_1\cap \mathcal C_{\delta_r}}|w_r| \leq C(n,\lambda)\delta_r^\alpha\|w_r\|_{L^2(B_2\setminus B_1)}$ by \Cref{lem:holderdecay};
                \item $\sup_{B_1\cap \mathcal C_{\delta_r}}|x_{n-1}^{\lambda-1}| \leq \delta_r^{\lambda-1}$ by definition of $\mathcal C_{\delta_r}$;
                \item $\mathcal H^{n-2}({\de B_1\cap \mathcal C_{\delta_r}\cap\{x_n=0\}})\leq C(n) \delta_r$ by definition of $\mathcal C_{\delta_r}$;
            \end{itemize}
            we obtain
            \begin{align*}
                -\frac{1}{c r^\lambda }\int_{B_1\cap\{u_r>0\}} & x_j\de_j w_r \lap w_r  \geq -\int_{\de B_1\cap \mathcal C_{\delta_r}\cap\{x_n=0\}}x_j^2|w_r||x_{n-1}|^{\lambda-1}\,d\mathcal H^{n-2}\\
                &\geq -C(n,\lambda)\sup_{B_1\cap \mathcal C_{\delta_r}}|x_{n-1}^{\lambda-1}w_r| \mathcal H^{n-2}({\de B_1\cap \mathcal C_{\delta_r}\cap\{x_n=0\}})\\
                &\geq -C(n,\lambda)\delta_r^{\alpha+\lambda}\|w_r\|_{L^2(B_2\setminus B_1)}.
            \end{align*}
            Recalling that $r^\lambda\delta_r^\lambda \leq C \sup_{B_{3/2}} |w_r| \le C \|w_r\|_{L^2(B_2\setminus B_1)}\leq Cr^{\lambda+\beta}$ we obtain \eqref{eq:gradwrlapwr}. 
        \end{proof}
        \begin{proof}[Proof for $\lambda$ odd]
            \Cref{eq:wrlapwr} follows immediately joining \Cref{lem:estimatewlapw} with assumption \eqref{eq:extradecayassumption}.
            
            The proof of \Cref{eq:gradwrlapwr} is very similar to the one for the half-odd-integer case. Since, for $\lambda$ odd, $\psi_\lambda$ vanish identically on the thin space, we have
            \begin{equation*}
                (x\cdot \nabla w_r)\lap w_r\equiv 0\text{ on }B_1\cap\{u_r=0\},
            \end{equation*}
            so \eqref{eq:gradwrlapwr} is trivial in the region $B_1\cap\{u_r=0\}$. On the other hand the region $B_1\cap\{u_r>0\}$ is again contained in the cilinder $\mathcal C_{\delta_r}$, by \Cref{lem:barrierodd}. Since on this set $\lap w_r(x)=c r^\lambda|x_{n-1}|^{\lambda-1}\delta_0(x_n)$, we can perform exactly the same computation of the case $\lambda\in \frac 32 + 2\N$.
        \end{proof}
        We can now prove the monotonicity of the truncated frequency at any truncation order $\gamma$.
        \begin{lemma}\label{lem:frequencymonotonicity}
            Assume \eqref{eq:extradecayassumption} holds for some $u,\lambda,\beta$ and $r_\circ$.
            Set $w_r:=(u-\psi_\lambda)(r\cdot)$. For all $\gamma\geq 0$ there exists $C=C(n,\lambda,\beta,\gamma, r_\circ)$ such that we have
            \begin{equation}\label{eq:frequencymonotonicity}
                \phi_\gamma(r,w)+g_\gamma(r,w)\leq C\text{ and } \frac{d}{dr}\phi_\gamma(r,w)\geq  \frac 2 r \frac{\left(\int_{B_1}w_r\lap w_r\right)^2}{H^2(r,w)+r^{2\gamma}} -Cr^{(\alpha\beta/\lambda)-1}
            \end{equation}
            for all $r\in(0,r_\circ)$ .
        \end{lemma}
        \begin{proof}
            Joining \Cref{lem:keyestimateformonotonicity} with \eqref{eq:formulaphigammaderivative} we find 
            \begin{equation}\label{eq:local1}
                \phi'_\gamma(r,w)\geq -C_0(\phi_\gamma(r,w)+1)g_\gamma(r,w)r^{\kappa-1}\text{ for all }r\in(0,r_\circ),
            \end{equation}
            where $\kappa:=\alpha\beta/\lambda>0$ and $C_0=C_0(n,\lambda)$. Now the proof is identical to Lemma 4.3 in \cite{FRS1}, let us sketch it.
            
            The key point is noticing that for $\gamma=0$ the statement is a simple consequence of \Cref{lem:ellipticestimates} and \eqref{eq:extradecayassumption}; and that joining \eqref{eq:local1} and part (1) in \Cref{lem:H ratio} one gets the following implication: If for some $\gamma\geq 0$ we have $C_\gamma>0$ such that
            $$
            \sup_{(0,r_\circ)}\phi_\gamma(\cdot,w)+g_\gamma(\cdot,w)\leq C_{\gamma},$$
            then exists $C_{\gamma+\kappa/3}>0$ depending on $C_\gamma,n,\lambda,\beta, r_\circ$, such that
            $$ \sup_{(0,r_\circ)}\phi_{\gamma+\kappa/3}(\cdot,w)+g_{\gamma+\kappa/3}(\cdot,w)\leq C_{\gamma +\kappa/3}.
            $$
            Iterating this reasoning a finite number of times one eventually ends up with \eqref{eq:frequencymonotonicity}.
        \end{proof}

        \Cref{lem:frequencymonotonicity} proves in particular that $\phi_\gamma(0+,w)$ exists finite for each $\gamma\geq 0$. We are ready to prove the decay improvement.

        \begin{proof}[Proof of \Cref{prop:decayimprovement} for $\lambda\in \frac 3 2 +2\N$]
            Let $w:=u-\psi_\lambda$, and fix $\gamma:=\lambda+2$. We start noticing that it is enough to show that
            \begin{equation}\label{eq:high freq is enough}
                \lambda^*:=\phi_{\gamma}(0+,w)\geq \lambda +1.
            \end{equation}
            Indeed, if \eqref{eq:high freq is enough} holds, then thanks to \Cref{lem:frequencymonotonicity} we can apply by part (2) of \Cref{lem:H ratio} with $f$, $\kappa$, and $\lambda_*$ respectively replaced by $w$, $\frac{\alpha\beta}{\lambda}$, and $ \lambda+1$, to obtain 
            \[
            \frac{H(r_\circ, w) + r_\circ^{2\gamma}}{H(r, w) + r^{2\gamma}} \ge C_3(r_\circ/r)^{2(\lambda+1)} \quad \mbox{for all }r\in (0,r_\circ)
            \]
            Combining this with \Cref{lem:ellipticestimates} we obtain \eqref{eq:decayimprovement}. 
            
            Now, in order to prove \eqref{eq:high freq is enough}, we set $\phi_{\gamma}(0^+,w)$ and assume by contradiction $\lambda^*<\lambda+1$.
            
            \textbf{Step 1.} Let us show that $r^{2\gamma}\ll H(r,w)\ll r^{2\lambda+\beta}$ as $r\downarrow 0$ and $\lim_{r\downarrow 0} \phi(r,w)=\lambda^*\geq \lambda+\beta/2.$
            
            Indeed, Since $\lambda^*<\lambda+1$, we must have $\phi_\gamma(\cdot,w)\leq \lambda+1$ in a sufficiently small interval $[0,r_1]$. Then part 1 of \Cref{lem:H ratio} ensures $H(r,w)+r^{2\gamma}\gg r^{2\lambda +2}$, which implies $H(r,w)\gg r^{2\gamma}$, because of the choice of $\gamma$. Hence we find
            \begin{equation*}
                \lambda^* = \lim_{r\downarrow 0} \phi_\gamma (r,w)=\lim_{r\downarrow 0}\frac{\phi(r,w)+\frac{\gamma r^{2\gamma}}{H_0(r,w)}}{1+\frac{r^{2\gamma}}{H_0(r,w)}}=\lim_{r\downarrow 0}\phi(r,w). 
            \end{equation*} By assumption \eqref{eq:extradecayassumption} we have $H(r,w)=O( r^{2\lambda+2\beta})$, so if we had $\lambda^*\leq \lambda+\beta/2$ we could use part 1 of \Cref{lem:H ratio} (with $\overline\lambda =\lambda+3\beta/4$ $\delta =\beta/4$) and find $r^{2\lambda+\beta}\ll H(r,w) = O(r^{2\lambda+2\beta})$ as $r\downarrow 0$; a contradiction.

            \textbf{Step 2.} Set $\widetilde w_r:=w_r/\|w_r\|_{L^2(\de B_1)}$. We now show that for some sequence $r_k\downarrow 0$, we have $\widetilde w_{r_k}\rightharpoonup Q$ weakly in $H^1_\loc(\R^n)$, where $Q$ is a nonzero solution of problem~\eqref{eq:thinobstaclelinearized}.
        
            Indeed, since $\lambda_*= \phi(0+,w)<\lambda+1$, we can fix $\overline\lambda\in (\lambda_*, \lambda+1)$. Now, let $\bar r>0$ be such that $\sup_{[0,\bar r]}\phi_{\gamma}(\cdot, w)\leq \overline\lambda$.
            Moreover, thanks to Step 1, we can always assume (for $H(r, w) \ge  r^{2\gamma}$ for all $r>0$ small enough)
            Let us fix $\delta>0$ so that $2\overline\lambda + \delta<2\lambda+2$.
            Hence, using part 1 of \Cref{lem:H ratio} we obtain: 
            \begin{align}\label{H123}
                H(R,\widetilde w_r)=\frac{H(Rr,w)}{H(r,w)}\leq 2\frac{H(Rr,w)+(Rr)^{2\gamma}}{H(r,w)+r^{2\gamma}}\leq 2 C_2 {R^{2\overline \lambda+\delta}},
            \end{align}
            for all $R\in(1,\bar r /r)$, for all $r>0$ small enough.
             
            We combine this with \Cref{lem:ellipticestimates} and find that $\{\widetilde w_r\}$ is bounded in $H^1_\loc(\R^n)$. So we extract a sequence $r_k\downarrow 0$ such that $w_{r_k}\weak Q$ weakly in $H^1_\loc(\R^n)$. By compactness of the trace operator, $\|Q\|_{L^2(\de B_1)}=1$.
            
            We now observe that $Q$ is a global solution of \eqref{eq:thinobstaclelinearized}. Indeed, fix some large $R>0$, by \Cref{lem:barrier} applied to $w_{Rr}$ and \eqref{eq:extradecayassumption} we find 
            \begin{equation*}
                	B_R\cap\{x_{n-1}\leq -\delta_{k,R}\}\subset\{ w_{r_k}=0\};\quad B_R\cap \supp\lap w_{r_k}\subset B_R\setminus\mathcal N_{\delta_{k,R}},
            \end{equation*}
            for some sequence $\delta_{k,R}$ which goes to zero if we let $r_k\downarrow 0$ leaving $R$ fixed. Of course the same holds for $\widetilde w_{r_k}$. Passing this information to the limit we obtain that $Q$ is a ($H^1_\loc(\R^n)$ weak) solution of \eqref{eq:thinobstaclelinearized}.

            \textbf{Step 3.} We now reach a contradiction using the information about growth at infinity and vanishing order at 0 on $Q$ and the classification of solutions to \eqref{eq:thinobstaclelinearized} given in in \Cref{prop:sphericalharmonics}.
            
            Indeed, on the one hand, passing  \eqref{H123} to the limit ---for $r=r_k$--- (note that $\tilde w_{r_k} \to Q$ in $L^2(\partial B_R)$ for all $R>1$) we obtain:
            \begin{equation}\label{ebuioguwi1}
                  H(R,Q) \le C R^{2\overline\lambda+\delta} \ll R^{2\lambda+2} \quad \mbox{(as $R\to \infty$}).
            \end{equation}
            
            On the other hand, reasoning similarly as in \eqref{H123}, but now using part 2 of \Cref{lem:H ratio} instead of part 1 we obtain: 
            \begin{equation}\label{H123hee}
                H(R,\widetilde w_r)=\frac{H(Rr,w)}{H(r,w)}\ge \frac 1 2 \frac{H(Rr,w)+(Rr)^{2\gamma}}{H(r,w)+r^{2\gamma}}\ge \frac{C_3}{2} {R^{2\lambda^*}},
            \end{equation}
            for all $R\in(1,r_\circ /r)$, whenever $Rr$ is small enough. Passing this to the limit we obtain:
            \begin{equation}\label{ebuioguwi2}
                  H(R,Q) \ge C R^{\lambda^*} \gg R^{2\lambda} \quad \mbox{(as $R\to \infty$}), 
            \end{equation}
            where $C>0$.

            Now,  using since solutions $Q$ of the linear problem \eqref{eq:thinobstaclelinearized}  were classified in \Cref{prop:sphericalharmonics}) and they are  sums of homogeneous functions with homogeneities in $\frac 1 2 +\N$, we see that the only way that  \eqref{ebuioguwi1}-\eqref{ebuioguwi2} may be satisfied at the same time is if $Q\equiv 0$. Since by construction it was $\|Q\|_{L^2(\partial B_1)}=1$ we reached a contradiction. 
        \end{proof}
        \begin{proof}[Proof of \Cref{prop:decayimprovement} for $\lambda$ odd]
            We can repeat the argument for $\lambda\in \frac 32 + 2\N$. The proof works verbatim until the end of Step 1. Step 2 holds except for the fact that the blow-up $Q$ is (the even reflection of) a harmonic function which vanish on the thin space. A similar contradiction is found in Step 3.
        \end{proof}
\section{Proof of the main results}\label{sec:proofmainresults}
    \Cref{prop:acceleratedecay} can be proved also for $\lambda\in\N$, with a similar method. However in the integer case stronger results are known, as explained next. 
    
    On the one hand, when $\lambda$ is odd, the following result of Savin and Yu has been recently proved.
    \begin{theorem}[{\cite[Theorem 1.1]{SY2}}]\label{thm:savinyu}
        Let $u$ solve \eqref{eq:thinobstacle} with $0\in\Lambda_\lambda(u)$ for some $\lambda\in\{1,3,5,\ldots\}$. Then there exist a constants $\beta\in(0,1)$ and $C>0$, depending only on $n$ and $\lambda$, such that
        \begin{equation*}\label{eq:extradecayassumptionodd}
            \sup_{B_r}|u-\tau \psi_\lambda\circ S|\leq Cr^{\lambda+\beta}\text{ for all }r\in(0,1),
        \end{equation*}
        for some $\tau >0$ and $S\in O(n-1)$.
    \end{theorem}
    
    On the other hand when $\lambda$ is even  then $\psi_\lambda$ is an harmonic polynomial is even and then $\phi(\cdot,u-\psi_\lambda)$ is increasing by a similar arguments as in \cite{FS}. This key observation leads to
    \begin{theorem}[{\cite[Proposition 1.10, items (ii) and (iii), and Remark 1.12]{FeJ}}]\label{thm:fernandezj}
        Let $u$ solve \eqref{eq:thinobstacle}. Then there exists $\Sigma_{\text{\upshape{even}}}\subset \Lambda(u)$ such that
        \begin{equation*}
            \dim_{\mathcal H} \Sigma_{\text{\upshape{even}}} \leq n-3,
        \end{equation*}
        and, for all $\lambda\in \{2,4,6,\ldots\}$ and $x_\circ\in \Lambda_\lambda(u)\setminus\Sigma_{\text{\upshape{even}}}$ it holds
        \begin{equation*}\label{eq:extradecayassumptioneven}
            \sup_{B_r}|u-q|\leq Cr^{\lambda+1}\text{ for all }r\in(0,1),
        \end{equation*}
        where $q$ is a $\lambda$-homogeneous solution of \cref{eq:thinobstacle} and $C=C(n,\lambda)$.
    \end{theorem}
    We are ready to prove the \Cref{thm:mainintro}.
    \begin{proof}[Proof of \Cref{thm:mainintro}]
        We define the set $\Sigma$ as
        \begin{equation}
            \Sigma(u):=\Lambda_{\text{\upshape{other}}}(u)\cup\Sigma_{\text{even}}(u)\cup\bigcup_{k\geq 1} \Lambda_k^*(u) \cup \bigcup_{k\geq 1}\Lambda_{2k-1/2}^*(u).
        \end{equation}
        Notice that
        \begin{itemize}
            \item $\dim_{\mathcal H} \Lambda_{\text{\upshape{other}}} \leq n-3$, by \Cref{prop:dimredotherfreq},
            \item $\dim_{\mathcal H} \Sigma_{\text{\upshape{even}}} \leq n-3$, by \Cref{thm:fernandezj},
            \item for each $\lambda\in \{\tfrac 3 2,\tfrac 7 2, \tfrac{11}2 \ldots\}\cup\{1,2,3,\ldots\}$, $\dim_{\mathcal H} \Lambda^*_\lambda \leq n-3$, by \Cref{prop:dimensionreduction},
        \end{itemize}
        hence $\dim_{\mathcal H} \Sigma\leq n-3$.

        Now, given some $x_\circ \in\Lambda(u)\setminus\Sigma(u)$, set $\lambda:=\phi(0+,u(x_\circ+\cdot)$. We want to prove \cref{eq:expansionintro}, there are only three cases.
        
        If $\lambda$ is even then the second part of \Cref{thm:fernandezj} gives \eqref{eq:expansionintro} for a certain $q$, because $x_\circ\in\Lambda_\lambda(u)\setminus\Sigma_{\text{\upshape{even}}}$. But on the other hand $x_\circ\in\Lambda_\lambda(u)\setminus\Lambda_\lambda^*(u)$, so \Cref{lem:closeness by frequency} gives that $q$ is 2D, i.e., $q=\tau\psi_\lambda\circ S$ for some $\tau\geq 0$ and $S\in O(n-1)$.
        
        If $\lambda$ is odd we apply \Cref{thm:savinyu} and obtain that assumption \eqref{eq:extradecayassumption} is satisfied. Hence we can apply \Cref{prop:decayimprovement} and obtain \eqref{eq:expansionintro}.
        
        If $\lambda$ is an half-odd-integer we can apply \Cref{prop:dichotomy} since $x_\circ\in\Lambda_\lambda(u)\setminus\Lambda_\lambda^*(u)$. This gives that assumption \cref{eq:extradecayassumption} holds and so we can apply \Cref{prop:decayimprovement} and obtain \eqref{eq:expansionintro}.
   \end{proof}
   We can also prove \Cref{cor:covering}.
   \begin{proof}[Proof of \Cref{cor:covering}]
       Let $u$ solve \eqref{eq:thinobstacle} and let $\lambda\in \{\tfrac 3 2,\tfrac 7 2,\ldots\}\cup\{1,2,3,4,\ldots\}$. We partition $B_{1/2}\cap(\Lambda_\lambda(u)\setminus\Sigma(u))=\cup_{N\geq 1}E_N$ where $x_\circ\in E_N$ if and only if
       \begin{equation*}
        \sup_{B_r}|u(x_\circ+\cdot)- q_{x_\circ}|\leq N r^{\lambda+1}\text{ for all }r\in(0,1/2),\quad \|q_{x_\circ}\|_{L^2(\partial B_1)}>\frac 1 N,
       \end{equation*}
       where $q_{x_\circ}$ is the 2D blow-up given by \Cref{thm:mainintro}. Let us work out case $\lambda\in\{\tfrac 3 2,\tfrac 7 2,\tfrac{11}{2},\ldots\}$, the other cases being completely analogous. We define $e_{\circ}\in\R^{n-1}$ with $|e_{x_\circ}|=1$ by the property
       \begin{equation*}
           q_{x_\circ}(x',0)= c_{x_\circ}(e_{x_\circ}\cdot x')_+^\lambda\text{ for all }x'\in\R^{n-1},
       \end{equation*}
       where $c_{x_\circ}$ is comparable to $\|q_{x_\circ}\|_{L^2(\partial B_1)}\ge \frac 1 N$.
       We claim that the map $ E_N\to \R\times\R^{n-1}$ defined by $x_\circ\mapsto (0,e_{x_\circ})$ is a $C^{1,1}$ Whitney data. This implies that there exist $f\in C^{1,1}(\R^{n-1})$ such that, for all $x_\circ\in E_N$, $f(x_\circ)=0$ and $\nabla f(x_\circ)=e_{x_\circ}$. Applying the implicit function theorem to $f$, we conclude that $E_N$ is locally covered by $C^{1,1}$ manifolds of dimension $n-2$.
       
       We prove the claim. Take any $x_0,x_1\in E_N$ and set $r_\circ:=3|x_0-x_1|$, without loss of generality we can assume $r_\circ<1/4$. By definition of $E_N$ and the triangular inequality we find
       \[
       \sup_{B_r}|q_{x_0}-q_{x_1}(\cdot -(x_0-x_1))|\leq \sup_{B_r}|u(x_0+\cdot)-q_{x_0}|+\sup_{B_{2r}}|u(x_1+\cdot)-q_{x_1}|\leq 3N r^{\lambda+1}.
       \]
       So, scaling and setting $b:=e_{x_1}\cdot(x_0-x_1)/r$ we get
       \[
       \sup_{x\in B_1}|c_{x_\circ}(e_{x_0}\cdot x')_+^\lambda - c_{x_1}(e_{x_1}\cdot x' - b)_+^\lambda|\leq 3N r,
       \]
       since by construction $|b|\leq 1/3$ this is easily implies \begin{equation}\label{eq:whitney}
           |e_{x_0}-e_{x_1}|+|b|\leq C N^2 r
       \end{equation} for some constant $C=C(n,\lambda)$. In fact, the map sending $\R\times\mathbb S^{n-2}\times \R \to C(\overline B_1)$ defined by $(\tau,e,b)\mapsto \tau(e\cdot x'-b)_+^\lambda$ is a bi-lipschitz homeomorphism with its image, when restricted to $[1/2,2]\times\mathbb S^{n-2}\times[-1/3,1/3]$. We can use the classical $C^1$ Whitney extension thanks to the bound on $b$ in \eqref{eq:whitney}. The resulting function will have Lipschitz gradient because of the bound on $|e_{x_0}-e_{x_1}|$ in \eqref{eq:whitney}. 
   \end{proof}

\end{document}